\documentclass[12pt]{amsart}
\usepackage{amsmath,amssymb,color,mathtools,listings,amsfonts,amsthm,stmaryrd,rotating,stmaryrd,enumitem}
\usepackage{axodraw2}
\usepackage{shuffle}
\usepackage{wasysym}
\input xy
\xyoption{all}
\newtheorem{example}{Example}
\newtheorem{remark}{Remark}
\newtheorem{definition}{Definition}
\newtheorem{proposition}{Proposition}
\newtheorem{lemma}{Lemma}
\newtheorem{theorem}{Theorem}
\newtheorem{corollary}{Corollary}
\usepackage{todonotes}
 
\usepackage[toc,page]{appendix}
 
\usepackage{tikz,tkz-tab, tikz-cd}
\usepackage{difftrees}
\tikzset{
  ncone/.pic={
	\draw (0,0)--(0,0.2);
  }
}

\tikzset{
  nctwo/.pic={
    \draw (0,0)--(0,0.2);
	\draw (0.1,0)--(0.1,0.2);
  }
}

\tikzset{
  nctwoW/.pic={
    \draw (0,0.2)--(0,0)--(0.1,0)--(0.1,0.2);
  }
}

\tikzset{
  nctwoWW/.pic={
    \draw (0,0.2)--(0,0)--(0.2,0)--(0.2,0.2);
  }
}

\tikzset{
  ncthreeWW/.pic={
    \draw (0,0.2)--(0,0)--(0.3,0)--(0.3,0.2);
	\draw (0.2,0)--(0.2,0.2);
  }
}

\tikzset{
  ncthree/.pic={
    \draw (0,0)--(0,0.2);
	\draw (0.1,0)--(0.1,0.2);
	\draw (0.2,0)--(0.2,0.2);
  }
}

\tikzset{
  ncthreeW/.pic={
    \draw (0,0.2)--(0,0)--(0.2,0)--(0.2,0.2);
	\draw (0.1,0)--(0.1,0.2);
  }
}

\tikzset{
  ncfour/.pic={
    \draw (0,0)--(0,0.2);
	\draw (0.1,0)--(0.1,0.2);
	\draw (0.2,0)--(0.2,0.2);
	\draw (0.3,0)--(0.3,0.2);
  }
}
\tikzset{
  ncfive/.pic={
    \draw (0,0)--(0,0.2);
	\draw (0.1,0)--(0.1,0.2);
	\draw (0.2,0)--(0.2,0.2);
	\draw (0.3,0)--(0.3,0.2);
	\draw (0.4,0)--(0.4,0.2);
  }
}

\tikzset{
  ncfourW/.pic={
    \draw (0,0.2)--(0,0)--(0.3,0)--(0.3,0.2);
	\draw (0.1,0)--(0.1,0.2);
	\draw (0.2,0)--(0.2,0.2);
  }
}
\tikzset{
  ncfiveW/.pic={
    \draw (0,0.2)--(0,0)--(0.4,0)--(0.4,0.2);
	\draw (0.1,0)--(0.1,0.2);
	\draw (0.2,0)--(0.2,0.2);
	\draw (0.3,0)--(0.3,0.2);
  }
}

\tikzset{
  nconeinsidetwoWW/.pic={
    \path (0,0) pic {nctwoWW}; \path (0.1,0.1) pic {ncone};
  }
}

\tikzset{
  nconeinsidethreeWW/.pic={
    \path (0,0) pic {ncthreeWW}; \path (0.1,0.1) pic {ncone};
  }
}
\tikzset{
  nconeinsidethreerightWW/.pic={
    \draw (0,0.2)--(0,0)--(0.3,0)--(0.3,0.2);
    \draw (0.1,0)--(0.1,0.2); \draw (0.2,0.1)--(0.2,0.3);
  }
}
\tikzset{
  nconeinsidethreeleftWW/.pic={
    \draw (0,0.2)--(0,0)--(0.3,0)--(0.3,0.2);
    \draw (0.2,0)--(0.2,0.2); \draw (0.1,0.1)--(0.1,0.3);
  }
}
\tikzset{
  nctwoWWW/.pic={
    \draw (0,0.2)--(0,0)--(0.3,0)--(0.3,0.2);
  }
}

\tikzset{
  nconeoneinsidetwoWW/.pic={
	\path (0,0) pic {ncone};
    \path (0.1,0) pic {nctwoWW}; 
	\path (0.2,0.1) pic {ncone};
  }
}

\long\def\ignore#1{}
\usepackage{scalerel}
\newcommand\fat[1]{\ThisStyle{\ooalign{%
  \kern.46pt$\SavedStyle#1$\cr\kern.33pt$\SavedStyle#1$\cr%
  \kern.2pt$\SavedStyle#1$\cr$\SavedStyle#1$}}}

\newcommand{\frontstick}{\,\raisebox{-1pt}{\begin{tikzpicture}
\draw [line width=1pt,] (0,0)--(0,0.25);
\end{tikzpicture}}\kern+2pt}

\newcommand{\ad}{\mop{ad}} 
\usepackage{forest}
\newcommand{\pl}{\ {\triangleright\hspace{-0.3cm}{-}}}
\forestset{
  decor/.style = {
    label/.expanded = {[inner sep = 0.2ex, font=\unexpanded{\tiny}]right:{$#1$}}
  },
  root/.style = {minimum size = 0.1ex},
  decorated/.style = {
    for tree = {
      circle, fill, inner sep = 0.2ex, minimum size = 0.5ex,
      grow' = north, l = 0, l sep = 1.0ex, s sep = 0.6em,
      fit = tight, parent anchor = center, child anchor = center,
      delay = {decor/.option = content, content =}
    }
  },
  default preamble = {decorated, root},
  begin draw/.code={\begin{tikzpicture}[baseline={([yshift=-0.5ex]current bounding box.center)}]},
}

\font \sevenrm=cmr7
\font \fiverm=cmr5
\newcommand{\nc}{\newcommand}
\nc\smsc{0.8}
\def \restr#1{\mathstrut_{\textstyle |}\raise-6pt\hbox{$\scriptstyle #1$}}
\def \srestr#1{\mathstrut_{\scriptstyle |}\hbox to
-1.5pt{}\raise-4pt\hbox{$\scriptscriptstyle #1$}}
\nc{\mop}[1]{\mathop{\hbox {\rm #1} }\nolimits}
\nc{\gmop}[1]{\mathop{\hbox {\bf #1} }\nolimits}

\nc{\smop}[1]{\mathop{\hbox {\sevenrm #1} }\nolimits}
\nc{\ssmop}[1]{\mathop{\hbox {\fiverm #1} }\nolimits}
\nc{\mopl}[1]{\mathop{\hbox {\rm #1} }\limits}
\def\dbar{d\hskip-3pt \raise 4pt\hbox{-}}
\nc{\smopl}[1]{\mathop{\hbox {\sevenrm #1} }\limits}
\nc{\ssmopl}[1]{\mathop{\hbox {\fiverm #1} }\limits}

 \newcommand{\lbutcher}{{\raise 3.9pt\hbox{$\circ$}}\hskip -1.9pt{\scriptstyle \searrow}\,}
 \newcommand{\rbutcher}{{\raise 3.9pt\hbox{$\circ$}}\hskip -1.9pt{\scriptstyle \swarrow}}
 \newcommand{\lbc}{{[\hskip -2.9pt [\hskip -2.9pt [\hskip -2.9pt [}}
  \newcommand{\rbc}{{]\hskip -2.9pt ]\hskip -2.9pt ] \hskip -2.9pt ]}}

\title[Connection algebra]{\small{Algebraic aspects of connections:\\ 
from torsion, curvature, and post-Lie algebras to Gavrilov's double exponential\\ and special polynomials}}

\author[M.~J.~H.~Al-Kaabi]{Mahdi J.~Hasan~Al-Kaabi}
\address{Mathematics Dept., College of Science, Mustansiriyah University, Baghdad, Iraq.}
\email{mahdi.alkaabi@uomustansiriyah.edu.iq}

\author[K.~Ebrahimi-Fard]{Kurusch Ebrahimi-Fard}
\address{Dept.~of Mathematical Sciences, NTNU, Trondheim, Norway.}
\email{kurusch.ebrahimi-fard@ntnu.no}
\urladdr{https://folk.ntnu.no/kurusche/}

\author[D.~Manchon]{Dominique Manchon}
\address{Lab.~de Math\'ematiques Blaise Pascal, CNRS--Universit\'e Clermont-Auvergne, France.}
\email{Dominique.Manchon@uca.fr}
\urladdr{https://lmbp.uca.fr/~manchon/}

\author[H.~Munthe-Kaas]{Hans Munthe-Kaas}
\address{Dept.~of Mathematics, University of Bergen, Bergen, Norway.}       
\email{hans.munthe-kaas@uib.no}
\urladdr{http://hans.munthe-kaas.no}

\begin{document}

\begin{abstract}
Understanding the algebraic structure underlying a manifold with a general affine connection is a natural problem. In this context, A.~V.~Gavrilov introduced the notion of framed Lie algebra, consisting of a Lie bracket (the usual Jacobi bracket of vector fields) and a magmatic product without any compatibility relations between them. In this work we will show that an affine connection with curvature and torsion always gives rise to a post-Lie algebra as well as a $D$-algebra. The notions of torsion and curvature together with Gavrilov's special polynomials and double exponential are revisited in this post-Lie algebraic framework. We unfold the relations between the post-Lie Magnus expansion, the Grossman--Larson product and the $K$-map, $\alpha$-map and $\beta$-map, three particular functions introduced by Gavrilov with the aim of understanding the geometric and algebraic properties of the double-exponential, which can be understood as a geometric variant of the Baker--Campbell--Hausdorff formula. We propose a partial answer to a conjecture by Gavrilov, by showing that a particular class of geometrically special polynomials is generated by torsion and curvature. This approach unlocks many possibilities for further research such as numerical integrators and rough paths on Riemannian manifolds. 
\end{abstract}



\maketitle


\tableofcontents


\section{Introduction}
\label{sec:intro}

Let $\mathcal M$ be a smooth manifold, and let $\mathcal{XM}=\mop{Der} C^\infty(\mathcal M)$ be the Lie algebra of vector fields on $\mathcal M$. An affine connection on $\mathcal M$ gives rise to a covariant derivative operator $\nabla$ on $\mathcal{XM}$. It is well-known that the induced binary product $\rhd$ defined by 
$$
	X\rhd Y:=\nabla_XY
$$ 
is left pre-Lie when the connection is flat and torsion free. This fact can be traced back to A.~Cayley's famous 1857 article on vector fields and rooted trees \cite{C1857}. See for example \cite{Burde2006, Manchon2011,McLMoMuVe2017}. An affine connection with constant torsion and vanishing curvature gives rise to a post-Lie algebra \cite{HA13}, a structure which appeared more recently in a paper by B.~Vallette on partition posets \cite{BV2007}. The closely related notion of $D$-algebra appeared independently in joint work by one of the present authors together with W.~Wright \cite{MKW08} in the context of numerical schemes for differential equations on Lie groups and homogeneous spaces. Pre-Lie algebras, also known as chronological algebras \cite{AG1981}, are rather well-studied in algebra, combinatorics and geometry. On the other hand, research on post-Lie and $D$-algebras is more recent. See for instance~\cite{AEM21, CEO20, ELM2015, FM2018, MQS2020, MunLun2011, HA13}.

\medskip

Understanding the algebraic structure underlying a manifold with a more general affine connection is a natural problem which attracted some attention \cite{Gavrilov2006, Gavrilov2008, Gavrilov2010, KN1963}. An obvious structure on the space $\mathcal{XM}$ is that of a so-called framed Lie algebra, consisting of a Lie bracket (the usual Jacobi bracket of vector fields) and a binary product (the right triangle $\rhd$) without any compatibility relations between them. This setting was studied in depth by A.~V.~Gavrilov in the 2006 work \cite{Gavrilov2006}. Our contribution aims at exhibiting a natural post-Lie algebra $\mathfrak g$ associated with these geometric data, and reinterpreting the main geometric notions in this framework, that is, torsion and curvature as well as Gavrilov's special polynomials and double exponential\footnote{The latter is a formal binary product on vector fields which can be understood as a geometric variant of the Baker--Campbell--Hausdorff formula. It is not associative in general, unless the connection is flat. Concrete examples can be found in applied differential geometry, for instance in the context of numerical schemes on manifolds \cite{GP2021}.}. To sum up, we advocate in the present paper the relevance of the post-Lie framework for a deeper and more refined understanding of Gavrilov's major findings on framed Lie algebras which are not necessarily post-Lie, including vector fields on a manifold endowed with an affine connection with curvature and torsion.\\

The paper is organized as follows. Section \ref{sec:postLie} is devoted to post-Lie and $D$-algebras. Several of the lengthier algebraic computations in the proofs of the algebraic statements in this section have been collected in Appendix \ref{app:proofs}. After the necessary background is recalled in Subsection \ref{ssec:reminders}, the free $D$-algebra as well as the free post-Lie algebra generated by a magmatic algebra $M$ are defined, in terms of the free unital associative algebra (the tensor algebra) $TM=\bigoplus_{k\ge 0}M^{{\textstyle\cdot} k}$ and the free Lie algebra $\mop{Lie}(M)$, respectively. The particular case of a free post-Lie and free $D$-algebra generated by a set $A$ is detailed in Paragraph \ref{ssec:freeDalg} using planar rooted trees and forests decorated by elements from $A$. In Appendix  \ref{ssec:multigrafting} we present briefly the notion of planar multi-grafting. Of particular interest is Equation \eqref{diamond-rhd} relating right Butcher product and left grafting on rooted trees. Reminders on the enveloping algebra of a post-Lie algebra (Subsection \ref{sect:enveloping}) and the post-Lie Magnus expansion (Subsection \ref{sec:plm}) are followed by a longer subsection on Gavrilov's $K$-map, $K: TM \to TM$, seen from a post-Lie algebra point of view (Subsection \ref{ssec:GavrilovK}). This map is re-interpreted as a linear automorphism of $TM$ mapping the Grossman--Larson product $*$ onto the defining product $\textstyle\cdot$ of $TM$ (Theorem \ref{thm:isoK}). An explicit non-recursive expression of the inverse map, $K^{-1}$, is given in terms of set partitions (Proposition \ref{Kinverse}). Paragraph \ref{par:Z} concludes this section by stating Equation \eqref{ZKchi} relating the $K$-map, the post-Lie Magnus expansion, $\chi$, and the logarithm $Z: y \mapsto \log^{\textstyle\cdot}\big(K(\exp^{\textstyle\cdot}(y)\big)$. 

\smallskip

Section \ref{sect:framed-pl} is devoted to Gavrilov's $\beta$-map \cite{Gavrilov2006}, obtained by pre-composing the aforementioned logarithm $Z$ with the canonical projection $p$ from the tensor algebra onto the enveloping algebra of a completed graded framed Lie algebra $(\widehat{\mathcal L},\rhd,\lbc.\,,.\rbc)$. This map is a keystone in the expression of Gavrilov's double exponential \cite{Gavrilov2006} considered in Section \ref{sec:double-exp}. A simple formula is obtained in terms of the $K$-map, the post-Lie Magnus expansion and the projection $p$ (Equation \eqref{beta-chi}).

\smallskip

In Section \ref{sec:covderiv}, we apply the mentioned algebraic results to the concrete setting of the framed Lie algebra $\mathcal{XM}$ of vector fields on a smooth manifold $\mathcal M$ endowed with an affine connection. Higher-order covariant derivatives are recast in the post-Lie framework in terms of the  $K$-map. We show in Proposition \ref{prop:g-post-lie} that the free Lie algebra $\mathfrak g$ (resp.~the tensor algebra $\mathcal A$ of $\mathcal{XM}$) over the ring $\mathcal R:=C^\infty(\mathcal M)$ is a post-Lie algebra (resp.~a $D$-algebra). As a consequence (Remark \ref{alt-beta}), we obtain in Equation \eqref{betachirho} an alternative expression of Gavrilov's $\beta$-map. The particular case of a flat connection with constant torsion, in which the framed Lie algebra $\mathcal{XM}$ itself is post-Lie, is detailed in Subsection~\ref{par:diff}.

\smallskip

Section \ref{sect:special} is devoted to Gavrilov's special polynomials \cite{Gavrilov2006}. We provide a partial answer to a conjecture put forward in  \cite{Gavrilov2006}, by showing that a natural (and rather broad) family of special polynomials can be expressed in terms of torsion, curvature and their higher-order covariant derivatives (Theorem \ref{thm:special}). An important intermediate result (Proposition \ref{prop:kernel}) expresses the kernel $\mathcal J$ of the action $\rho$ of $\mathfrak g$ by derivations on $C^\infty(\mathcal M)$ (which is an ideal for the Grossman--Larson Lie bracket) in terms of the so-called curvature elements, denoted $s(a.b)$ and introduced in Definition \ref{torsion-curvature-elts}. We also show that the kernel $\mathcal K$ of the action $\rhd$ of $\mathfrak g$ on $\mathcal{XM}$ is a Grossman--Larson ideal included in the kernel $\mathcal J$, and exhibit a nonzero element of $\mathcal K$ by means of the first Bianchi identity. The inclusion $\mathcal J \subset \mathcal K$ is in general strict, manifesting the presence of curvature.

\smallskip

Finally, Section \ref{sec:double-exp} discusses Gavrilov's double exponential, which we first describe heuristically by comparison of consecutive parallel transports. Then we express it in precise terms using the post-Lie Magnus expansion (Theorem \ref{thm:q}). 

\smallskip

We close the paper with a short synthesis of the results followed by a brief outlook\\

\noindent{\textbf{Acknowledgements:} This work was partially supported by the project \textsl{Pure Mathematics in Norway}, funded by Trond Mohn Foundation and Troms{\o} Research Foundation. MJHAK was funded by the Iraqi Ministry of Higher Education and Scientific Research and he would like to thank Mustanisiriyah University\footnote{https://uomustansiriyah.edu.iq}, College of Science, Mathematics Department for their continued support to carry out this work. KEF and HMK are supported by the Research Council of Norway through project 302831 \textsl{Computational Dynamics and Stochastics on Manifolds} (CODYSMA). KEF would also like to thank the Department of Mathematics at the Saarland University for warm hospitality during a sabbatical visit. DM is supported by Agence Nationale de la Recherche, projet \textsl{Combinatoire alg\'ebrique, r\'esurgence, probabilit\'es libres et op\'erades} CARPLO ANR20-CE40-0007.}

\section{Post-Lie and $D$-algebras}
\label{sec:postLie}

\subsection{Reminders on post-Lie and $D$-algebras}
\label{ssec:reminders}

Let $\mathbf k$ be a field of characteristic zero, which will be the real numbers $\mathbb R$ whenever differential geometry comes into play.

\begin{definition}\cite{BV2007}
A post-Lie algebra on $\mathbf k$ is a Lie algebra $(\mathfrak{g},[.\,,.])$ together with a bilinear mapping $\rhd$ : $\mathfrak{g} \times \mathfrak{g} \longrightarrow \mathfrak{g}$ compatible with the Lie bracket, in the following sense:
\begin{equation}\label{eq1}
	x \rhd [y,z] = [x\rhd y,z]+ [y, x \rhd z]
\end{equation}
\begin{equation}\label{eq2}
	[x,y]\rhd z = {\mathrm{a}}_{\rhd}(x,y,z) - {\mathrm{a}}_{\rhd}(y,x,z),
\end{equation}
for any elements $x, y, z \in \mathfrak{g}$. Here ${\mathrm{a}}_{\rhd}(x,y,z)$ is the associator with respect to the product $\rhd$ defined by:
$$
	{\mathrm{a}}_{\rhd}(x,y,z) := x \rhd (y \rhd z) - (x \rhd y) \rhd z.
$$
\end{definition} 

Any Lie algebra can be seen as a post-Lie algebra by setting the second product $\rhd$ to zero. Another possibility is to take for the second product $\rhd$ the opposite of the Lie bracket. A (left) pre-Lie algebra is just an Abelian post-Lie algebra, i.e., a post-Lie algebra with trivial Lie bracket, implying that \eqref{eq2} reduces to the (left) pre-Lie identity
$$
	0= {\mathrm{a}}_{\rhd}(x,y,z) - {\mathrm{a}}_{\rhd}(y,x,z).
$$
In other words, for a (left) pre-Lie algebra the associator is symmetric in the first two entries. On any post-Lie algebra, particular combinations of the Lie bracket and the $\rhd$ product yield two other operations, as follows:
\allowdisplaybreaks
\begin{align}
	\llbracket x, y\rrbracket & := x \rhd y - y \rhd x + [x, y], \label{twoLie}\\
	x \RHD y & := x \rhd y + [x, y], \nonumber
\end{align}
for all $x, y \in \mathfrak{g}$. From \eqref{eq1} and \eqref{eq2} above, one can easily deduce that $\big(\mathfrak{g}, \llbracket.\,, .\rrbracket\big)$ forms a Lie algebra, and the triple $\big(\mathfrak{g}, - [.\,, .], \RHD\big)$ is another post-Lie algebra \cite{CEO20, HA13} sharing the same double Lie bracket:
\allowdisplaybreaks
\begin{align*}
	\llbracket x, y\rrbracket 
	&= x \rhd y - y \rhd x + [x, y]\\
	&=x \RHD y - y \RHD x - [x, y]\\
	&=x\rhd y-y\RHD x.
\end{align*}

\begin{example}
\label{exm:connection}\rm
Let $\mathcal{XM}$ be the space of vector fields on a smooth manifold $\mathcal{M}$, which is equipped with an affine connection. For vector fields $X,Y \in \mathcal{XM}$, the covariant derivative of $Y$ in the direction of $X$ is denoted $\nabla_X Y =: X \rhd Y$. This defines an $\mathbb{R}$-linear, non-associative binary product on $\mathcal{XM}$. The torsion $t$ is {defined by}
\begin{align}
\label{torsion}
	t(X,Y) := X\rhd Y - Y\rhd X - \lbc X,Y \rbc,
\end{align}
where the bracket $\lbc .\,, . \rbc$ on the right is the usual Jacobi bracket of vector fields. It admits a covariant differential $\nabla t$. The curvature tensor $r$ is given by
\[
	r(X,Y)Z := X\rhd(Y\rhd Z) - Y\rhd(X\rhd Z) 
	- \lbc X,Y \rbc \rhd Z.
\]
It is known that for a flat connection with constant torsion, $r=0=\nabla t$, we have that $\big(\mathcal{XM},-t( . \, , . ),\rhd\big)$ defines a post-Lie algebra. The first Bianchi identity (see \eqref{bianchi-one} below) shows that $-t( . \, , . )$ obeys the Jacobi identity; skew-symmetry of $t$ implies anti-symmetry. Flatness is equivalent to identity $(\ref{eq2})$, whereas property $(\ref{eq1})$ follows from the definition of the covariant differential of $t$:
\[
	0 = (\nabla_X t) (Y,Z) = X \rhd t(Y,Z) 
						- t(Y,X \rhd Z) 
							- t(X \rhd Y,Z).
\]
\end{example}

\medskip

\begin{definition} \cite{MKW08}
\label{def:Dalgebra}
Let $(D,\fat{\cdot},\rhd)$ be an associative algebra with product~$m_D(u \otimes v)= u \fat{\cdot} v$ and unit $\mathbf{1}$, carrying another product $\rhd: D \otimes D \rightarrow D$ such that $\mathbf{1} \rhd v =v$ for all $v \in D$. Let 
$$
	\mathfrak d (D)
	:=\{u \in D \ |\ u \rhd (v \fat{\cdot} w) 
	= (u \rhd v) \fat{\cdot} w + v \fat{\cdot} (u \rhd w),\ \forall v,w \in D\}.
$$

The triple $(D,\fat{\cdot},\rhd)$ is called a $D$-algebra if the algebra product $\fat{\cdot}$ generates $D$  from $\{\mathbf{1},\mathfrak d (D)\}$ and furthermore for any $x \in \mathfrak d (D)$  and $v,w \in D$
\allowdisplaybreaks
\begin{align}
\label{D1}
	v \rhd x &\in \mathfrak d (D)\\
\label{D2}
	(x \fat{\cdot} v) \rhd w &= {\mathrm{a}}_{\rhd}(x,v,w).
\end{align}  
\end{definition} 

\begin{lemma}
\label{lem:DpostLie}
Let $(D,\fat{\cdot},\rhd)$ be as in Definition \ref{def:Dalgebra}. Then $\mathfrak d (D)$ together with $\rhd$ and Lie bracket $[u,v]:=u \fat{\cdot} v - v \fat{\cdot} u$ is a post-Lie algebra.
\end{lemma}

\begin{proof}
Remark that $\mathfrak d(D)$ is the set of $x\in D$ such that $L^\rhd_x :=x\rhd -$ is a derivation for the associative product. One only needs to verify that $\mathfrak d (D)$ is stable under the Lie bracket $(x,y)\mapsto x \fat{\cdot}y-y \fat{\cdot} x$, which amounts to prove that $L^\rhd_{x\fat{\cdot}y-y \fat{\cdot} x}$ is a derivation. From
$$
	L^\rhd_{x \fat{\cdot} y}=L^\rhd_x \circ L^\rhd_y - L^\rhd_{x \rhd y},
$$
which is a reformulation of \eqref{D2}, we get
$$
	L^\rhd_{x \fat{\cdot} y-y \fat{\cdot} x} = [L^\rhd_x, L^\rhd_y] - L^\rhd_{x \rhd y -y \rhd x},
$$
which proves the claim. Identity \eqref{eq1} results from the fact that any derivation for the associative product is also a derivation for the Lie bracket, and \eqref{eq2} is checked immediately.
\end{proof}

\begin{example}\label{example2}\rm
The space $\mathcal{DM}$ of differential operators on the manifold $\mathcal{M}$ endowed with an affine connection $\nabla$ with vanishing curvature and constant torsion is a $D$-algebra such that $\mathfrak d(\mathcal{DM})=\mathcal{XM}$.
\end{example}

\noindent This non-trivial example is treated in detail in \cite{HA13}. We shall revisit it in Subsection \ref{par:diff} below.

\subsection{Free post-Lie and free $D$-algebras}
\label{sec:Dtensoralg}

\subsubsection{The free $D$-algebra generated by a magmatic algebra}
\label{par:freeD}

Recall that a magmatic algebra consists of a set equipped with a binary operation and no further relations.

\begin{theorem}
\label{thm:Dtensor}
Let $(M, \rhd)$ be any magmatic algebra. Let $T(M)$ be the tensor algebra over $M$ with concatenation as product. Extending the magma product $\rhd$ to $T(M)$ as follows
\allowdisplaybreaks
\begin{align}
\label{Dtensor1}
	x \rhd (V W) &= (x \rhd V) W +
	V (x \rhd W)  \\
\label{Dtensor2}
	(xV) \rhd W &= {\mathrm{a}}_{\rhd}(x,V,W)
\end{align} 
for any element $x\in M$ and $V,W\in T(M)$, defines a $D$-algebra structure on $T(M)$.
\end{theorem}

\begin{proof}
It is clear that \eqref{Dtensor1} and \eqref{Dtensor2} uniquely define the extended magma product $\rhd$, by induction on the lengths of the elements involved. The tensor algebra is a Hopf algebra with the usual unshuffle coproduct, $\Delta_\shuffle(A):=A_{(1)} \otimes A_{(2)}$ (we use Sweedler's notation). The elements in the magmatic algebra $M$ are primitive, i.e., $\Delta_\shuffle(x)=x \otimes \mathbf{1} + \mathbf{1} \otimes x$ for all $x \in M$. A more explicit formula for the unshuffle coproduct is given by
\begin{equation}
\label{deshuffle-expl}
	\Delta_\shuffle(x_1\cdots x_n)=\sum_{I\sqcup J=\{1,\ldots,n\}}x_I\otimes x_J,
\end{equation}
with the obvious word notation $x_I:=x_{i_1}\cdots x_{i_p}$ for index set $I=\{i_1,\ldots,i_p\}$ with $i_1<\cdots <i_p$. The set $\mop{Prim} T(M)$ of primitive elements is actually the free Lie algebra $\mop{Lie}(M)$ generated by $M$. This is the vector subspace of $T(M)$ generated by iterated Lie brackets of elements of $M$ \cite{Ch.R93}. An iteration of \eqref{Dtensor1} yields
\begin{equation}
\label{Dtensor1-iterated}
	U\rhd(VW)=(U_{(1)}\rhd V)(U_{(2)}\rhd W),
\end{equation}
which is easily checked on monomials $U=x_1\cdots x_n$ (with $x_1,\ldots, x_n\in M$) by induction on the length $n$. We deduce from \eqref{Dtensor1-iterated} that the set of $U\in T(M)$ such that $L^\rhd_U$ is a derivation is precisely $\mop{Lie}(M)$. It remains to show that \eqref{Dtensor2} is valid for any element $x\in\mop{Lie}(M)$. It suffices to check that the set of elements of $T(M)$ verifying \eqref{Dtensor1} and \eqref{Dtensor2}, which contains $M$, is a Lie subalgebra. Let us choose two elements $X$ and $Y$ in $T(M)$ verifying \eqref{Dtensor1} and \eqref{Dtensor2}. The claim follows from a straightforward computation detailed in Appendix \ref{app:proofs}. 
\end{proof}

\begin{remark}\label{rmk:freeD}\rm
It is easily checked that $T(M)$ is the free $D$-algebra generated by the magmatic algebra $(M,\rhd)$. In fact, for any $D$-algebra $(D,\diamond,\cdot)$, any magmatic morphism $\psi: M \to (D,\diamond)$ can be uniquely extended to an associative algebra morphism $\Psi:T(M)\to D$, which happens to be a $D$-algebra morphism. Similarly, $\mop{Lie}(M)$ is the free post-Lie algebra generated by the magmatic algebra $M$.
\end{remark}

It is easily seen that $L_x^\rhd(y):=x\rhd y$ is a coderivation with respect to the coproduct $\Delta_\shuffle$ for any $x\in M$. More generally the coproduct is compatible with the extended magmatic product:

\begin{proposition}
\label{magma-deshuffle}
For any $U,V\in T(M)$ we have
$$
	\Delta_\shuffle(U\rhd V)
	=\Delta_\shuffle(U)\rhd\Delta_\shuffle(V)
	=U_{(1)}\rhd V_{(1)}\otimes U_{(2)}\rhd V_{(2)}.
$$
\end{proposition}

\begin{proof}
We can suppose that $U$ is a monomial, and we proceed by induction on its length $\ell$. The details are given in Appendix \ref{app:proofs}. 
\end{proof}

\noindent We mention the following result for later use:

\begin{proposition}
\label{prop:pre-glp}
For any $U,V,W\in T(M)$ the following holds:
\begin{equation}
\label{Dtensor2-iter}
	U\rhd(V\rhd W)=\big(U_{(1)}(U_{(2)}\rhd V)\big)\rhd W.
\end{equation}
\end{proposition}

\begin{proof}
We can suppose without loss of generality that $U$ is a monomial. We proceed by induction on the length $\ell$ of $U$. The details are given in Appendix \ref{app:proofs}. 
\end{proof}

\begin{definition}
\label{def:GLprod}
The Grossman--Larson product on $T(M)$ is defined by
\begin{equation}
\label{glp}
	U\ast V:=U_{(1)}(U_{(2)}\rhd V).
\end{equation}
\end{definition}

\noindent It is easily seen to be associative:
\allowdisplaybreaks
\begin{align*}
	(U\ast V)\ast W
		&=U_{(1)}(U_{(2)}\rhd V_{(1)})\Big(\big(U_{(3)} (U_{(4)}\rhd V_{(2)})\big)\rhd W\Big)\\
		&=U\ast(V\ast W).
\end{align*}
This is a straightforward computation using the cocommutativity of the unshuffle coproduct, left to the reader (see the proof of Proposition 3.3 in \cite{ELM2015}).

\begin{proposition}
\label{gl-deshuffle}
The Grossman--Larson product \eqref{glp} is compatible with the unshuffle coproduct: for any $U,V\in T(M)$ we have
$$
	\Delta_\shuffle(U\ast V)
	=\Delta_\shuffle(U)\ast\Delta_\shuffle(V)
	=U_{(1)}\ast V_{(1)}\otimes U_{(2)}\ast V_{(2)}.
$$
\end{proposition}

\begin{proof}
This is a straightforward check using \eqref{glp} and Proposition \ref{magma-deshuffle}.
\end{proof}


\subsubsection{The free $D$-algebra generated by a set. Planar rooted trees and grafting}
\label{ssec:freeDalg}

The following is immediate in view of Remark \ref{rmk:freeD}:

\begin{proposition}
Let $A$ be a finite alphabet and $\big(\mop{Mag}(A),\rhd\big)$ the free magmatic algebra over $A$. Then $\mop{Dalg}(A)=T\big(\mop{Mag}(A)\big)$, with the product $\rhd$ extended as in Theorem \ref{thm:Dtensor}, is the free $D$-algebra generated by $A$. Similarly, $\mop{PostLie}(A)=\mop{Lie}\big(\mop{Mag}(A)\big)$ is the free post-Lie algebra generated by $A$.
\end{proposition}

It is known that the free magma generated by $A$ can be represented in terms of planar rooted trees
$$
	T^{pl}_A=\bigg\{\Forest{[a]},
	\hskip 4mm 
	\Forest{[a[b]]},
	\hskip 4mm 
	\Forest{[a[b[c]]]},\ \Forest{[a[b][c]]},
	\hskip 4mm 
	\Forest{[a[b[c[d]]]]},\ \Forest{[a[b[c][d]]]},\ \Forest{[a[b][c[d]]]},\ \Forest{[a[b[c]][d]]},\ \Forest{[a[b][c][d]]}, \ldots \bigg\}
$$  
with nodes decorated by elements of $A$. Here the magmatic product $\diamond : T^{pl}_A \times T^{pl}_A \to T^{pl}_A$ is the right Butcher product defined as follows: $\sigma\diamond\tau$ is the $A$-decorated planar rooted tree obtained by grafting $\sigma$ on the root of $\tau$ on the right, for example
$$
	\Forest{[d[e]]}\diamond\Forest{[c[a][b]]}=\Forest{[c[a][b][d[e]]]}\ .
$$ 
Then $(T^{pl}_A,\diamond)$ is the free magma generated by $A$ via  the inclusion $A  \hookrightarrow T^{pl}_A$ given by $a \mapsto \Forest{[a]}$. Now, let $\mathcal T^{pl}_A$ denote the vector space freely spanned by $A$-decorated planar rooted trees. The {left}-grafting $\curvearrowright : \mathcal T^{pl}_A \times \mathcal T^{pl}_A \to \mathcal T^{pl}_A$ is defined as the $\mathbf{k}$-linear product given by
\begin{equation}\label{diamond-rhd-init}
	\tau \curvearrowright \Forest{[]}_a:= \tau\diamond\Forest{[a]}
\end{equation}
and 
\begin{equation}
\label{diamond-rhd}
	\tau \curvearrowright (\tau_1 \diamond \tau_2) 
	= (\tau \curvearrowright \tau_1) \diamond \tau_2
		+ \tau_1 \diamond (\tau \curvearrowright \tau_2)
\end{equation}
for $\tau,\tau_1,\tau_2 \in T^{pl}_A$. For example, $\Forest{[b[a]]}=\Forest{[b]} \curvearrowright \Forest{[a]} = \Forest{[b]} \diamond\Forest{[a]}$ and 
\allowdisplaybreaks
\begin{eqnarray*}
\Forest{[b[a]]}\curvearrowright \Forest{[e[c][d]]}
	&=&\Forest{[b[a]]}\curvearrowright\left(\Forest{[d]}\diamond\Forest{[e[c]]} \right)\\
	&=&\left(\Forest{[b[a]]}\curvearrowright\Forest{[d]}\right)\diamond\Forest{[e[c]]}
		+ \Forest{[d]}\diamond\left(\Forest{[b[a]]}\curvearrowright\Forest{[e[c]]}\right)\\
	&=&\Forest{[d[b[a]]]}\diamond\Forest{[e[c]]}+\Forest{[d]}\diamond\left(\Forest{[e[b[a]][c]]}
		+\Forest{[e[c[b[a]]]]}\right)=\Forest{[e[c][d[b[a]]]]]}+\Forest{[e[b[a]][c][d]]}+\Forest{[e[c[b[a]]][d]]}\ .
\end{eqnarray*}


\begin{lemma}
\label{lem:iso}
Let $\varphi: \mop{Mag}(A)\to \mathcal T^{pl}_A$ be the unique linear map defined by 
$\varphi(a):=\Forest{[a]}$ for all $a \in A$ and $\varphi(\sigma \rhd \tau) := \varphi(\sigma)\curvearrowright \varphi(\tau)$ for all $\sigma,\tau \in\mop{Mag}(A)$. The map $ \varphi$ is an isomorphism, i.e., the magmatic algebras $(\mathcal T^{pl}_A,\curvearrowright )$ and $(\mathrm{Mag}(A), \rhd)$ are isomorphic.
\end{lemma}

\begin{proof}
We have just seen that the correspondence $\kappa$, defined by $\kappa(a)=\Forest{[a]}$ for all $a\in A$ and
$$
	\kappa(\tau_1 \rhd \tau_2)=\kappa(\tau_1)\diamond\kappa(\tau_2)
$$
is a magma isomorphism. This is obviously still true for the correspondence $\overline\kappa$ analogously defined with the right Butcher product $\diamond$ replaced by its left version $\lbutcher$, where $\tau_1 \lbutcher \tau_2$ is the $A$-decorated planar rooted tree obtained by grafting $\tau_1$ on the root of $\tau_2$ on the left, for example
$$
	\Forest{[c[a][b]]}\lbutcher\Forest{[d[e]]}=\Forest{[d[c[a][b]][e]]}\ .
$$ 
Now both magmatic algebras $(\mathcal T^{pl}_A,\lbutcher)$ and $(\mathcal T^{pl}_A,\curvearrowright)$ are isomorphic. Indeed, the unique morphism $\Psi: (\mathcal T^{pl}_A,\lbutcher)\to (\mathcal T^{pl}_A,\curvearrowright)$ extending the identity on one-vertex trees can be put in upper-triangular matrix form with $1$'s on the diagonal, hence is a linear isomorphism. Full details are given in \cite{AM2014}. The isomorphism $\varphi$ is therefore given by $\varphi=\Psi\circ\overline\kappa$.
\end{proof}

\noindent As an example, we consider the tree
$$
\Forest{[c[a][b]]}
	= \Forest{[a]} \curvearrowright (\Forest{[b]} \curvearrowright \Forest{[c]}) 
		- (\Forest{[a]} \curvearrowright \Forest{[b]}) \curvearrowright \Forest{[c]}.
$$
Then we have 
\allowdisplaybreaks
\begin{align*}
	\varphi^{-1}\left(\Forest{[c[a][b]]}\right)
	&= \varphi^{-1}\big(\Forest{[a]} \curvearrowright (\Forest{[b]} \curvearrowright \Forest{[c]}) 
		- (\Forest{[a]} \curvearrowright \Forest{[b]}) \curvearrowright \Forest{[c]}\big)\\
	&=a \rhd (b  \rhd c) 
		-  (a \rhd b)  \rhd c\\
	&=\mathrm{a}_\rhd(a,b,c).	 
\end{align*}

\begin{remark}\label{lg}\rm
In view of Lemma \ref{lem:iso}, the free magmatic algebra $(\mop{Mag}(A),\rhd)$ can be represented by the linear span of planar $A$-decorated rooted trees, endowed by either the right Butcher product $\diamond$ or the left grafting $\curvearrowright$.
\end{remark}

\subsection{The enveloping algebra of a post-Lie algebra}
\label{sect:enveloping}

\begin{proposition}
\label{prop:quotient}
If $(\mathfrak{g}, [.\, ,.],\rhd)$ is a post-Lie algebra, then $\big(\mathcal U(\mathfrak{g}),\cdot,\rhd\big)$ is a $D$-algebra, and the set of primitive elements $\mathfrak{g}=\mop{Prim}\mathcal U(\mathfrak{g})$ is the post-Lie algebra $\mathfrak d \big(\mathcal U(\mathfrak{g})\big)$ of the $D$-algebra $\mathcal U(\mathfrak g)$.
\end{proposition}

\begin{proof}
Let us consider the $D$-algebra structure on $T(\mathfrak g)$ given in Paragraph \ref{par:freeD}. The ideal $J$ generated by $\{\jmath_{x,y}:=x.y-y.x-[x,y],\, x,y\in \mathfrak g\}$ is also a two-sided ideal for the product $\rhd$. To see this, choose any $x,y\in\mathfrak g$ and any $U,A,B$ in $T(\mathfrak g)$. By iterating \eqref{Dtensor1} we have
$$
	U\rhd(A\cdot\jmath_{x,y}\cdot B)
	=(U_{(1)}\rhd A)\cdot \jmath_{\scriptscriptstyle{U_{(2)}\rhd x,\, U_{(3)}\rhd y}} 
	\cdot (U_{(4)}\rhd B)\in J,
$$
using Sweedler's notation for the iterated deshuffle coproduct. Starting from $\jmath_{x,y}\rhd U=0$ which is a simple consequence of \eqref{Dtensor2}, we also have $(\jmath_{x,y}\cdot B)\rhd U=0$ by Proposition \ref{Dtensor2-iter}, taking primitiveness of $\jmath_{x,y}$ into account. Proposition \ref{Dtensor2-iter} also proves $(A\cdot \jmath_{x,y}\cdot B)\rhd U\in J$ by induction on the length of $A$. Hence the $D$-algebra structure on $T(\mathfrak g)$ gives naturally rise to a $D$-algebra structure on the quotient $\mathcal U(\mathfrak g)=T(\mathfrak g)/J$.
\end{proof}

\begin{remark}\rm
\label{rmk:adjfunctors}
From this follows the existence of a pair of adjoint functors between the categories of $D$-algebras and post-Lie algebras
$$
	\mathcal U(-): \text{postLie}\ \leftrightarrows \ D\text{-algebra}:\mathfrak g (-).
$$ 
In other words, there is a natural isomorphism 
$$
	\mathrm{Hom}_{\scriptscriptstyle{\text{postLie}}}\big(\mathfrak g (A),B\big)
	\rightarrow 
	\mathrm{Hom}_{\scriptscriptstyle{D\text{-algebra}}}\big(A, \mathcal U(B)\big).
$$
\end{remark}

\begin{remark}\rm
The Grossman--Larson product also makes sense on $\mathcal U(\mathfrak g)$, and is also given by \eqref{glp}. From Proposition \ref {gl-deshuffle}, it is compatible with the coproduct, making $\big(\mathcal U(\mathfrak g),\ast,\Delta)$ a Hopf algebra isomorphic to the enveloping algebra of $(\mathfrak g, \llbracket .\,,.\rrbracket)$. This has been first established in \cite{ELM2015}, see Proposition 3.3 therein. From Proposition \ref{prop:pre-glp} and \eqref{glp} we have
\begin{equation}
\label{product2}
	U\rhd(V\rhd W)=(U\ast V)\rhd W
\end{equation}
for any $U,V,W\in\mathcal U(\mathfrak g)$.
\end{remark}

Note however that both Hopf algebras have different antipodes. As an example, we compare the two antipodes for product $x \cdot y$ and $x \ast y$, $x,y \in \mathfrak{g}$
\allowdisplaybreaks
\begin{align*}
	S_*(x \cdot y)
	&= -x \cdot y - S_*(x)\ast y - S_*(y)\ast x \\
	&= y \ast x + x \rhd y\\
	&= y \cdot x + x \rhd y + y \rhd x \\
	&=S(x\cdot y)+ x \rhd y + y \rhd x .
\end{align*}
and 
\allowdisplaybreaks
\begin{align*}
	S(x \ast y)
	&= -x \ast y - S(x) \cdot y - S(y) \cdot x \\
	&=  y \cdot x - x \rhd y\\
	&= y \ast x - x \rhd y - y \rhd x \\
	&= S_*(x \ast y) - x \rhd y - y \rhd x .
\end{align*}

\noindent The following theorem is key to many upcoming computations.

\begin{theorem}
\label{thm:productantipode}
In the Hopf algebra $\big(\mathcal U(\mathfrak{g}),\cdot,\Delta_{\scriptscriptstyle{\shuffle}},\epsilon,S\big)$ the product can be expressed in terms of the Grossman--Larson product \eqref{glp} as follows:
\begin{equation}
\label{productnew}
	A \cdot B = A_{(1)} \ast \big( S_*(A_{(2)}) \rhd B\big), \qquad A,B \in U(\mathfrak{g}).
\end{equation} 
\end{theorem}

\begin{proof}
We use \eqref{glp} on the righthand side of \eqref{productnew}, which gives
\allowdisplaybreaks
\begin{align*}
	A_{(1)} \ast \big( S_*(A_{(2)}) \rhd B\big)
	&=A_{(1)(1)} \cdot \Big( A_{(1)(2)}   \rhd  \big(S_*(A_{(2)}) \rhd B\big)\Big)\\
	&\stackrel{\eqref{product2}}{=}A_{(1)(1)} \cdot \Big( \big(A_{(1)(2)} \ast S_*(A_{(2)})\big) \rhd B\Big)\\
	&=A_{(1)} \cdot \Big( \big(A_{(2)(1)} \ast S_*(A_{(2)(2)})\big) \rhd B\Big)\\
	&=A_{(1)} \cdot \big( m_\ast(\text{id} \otimes S_*)\Delta_{\scriptscriptstyle{\shuffle}} (A_{(2)}) \rhd B\big)\\
	&=A \cdot  B.
\end{align*}
In the third equality we used coassociativity. 
\end{proof}

\begin{remark}\rm
\label{rmk:GLproductrecursion}
Using \eqref{Dtensor1-iterated} and $x \in \mathfrak{g} \hookrightarrow \mathcal U(\mathfrak{g})$ being primitive, i.e., $\Delta_{\scriptscriptstyle{\shuffle}}(x)=x\otimes 1 + 1 \otimes x$, implying $S_\ast(x)=-x$, we find from \eqref{productnew} the recursion
\allowdisplaybreaks
\begin{align}
	x_1 \cdots x_n  
	&= x_1 * (x_2 \cdots x_n) - x_1 \rhd (x_2 \cdots x_n) \nonumber \\
	&\stackrel{\eqref{glp} }{=}x_1 * (x_2 \cdots x_n) 
	- \sum_{i=2}^n x_2 \cdots (x_1 \rhd x_i) \cdots x_n. \label{RecursionL1} 
\end{align} 
We used that for $x \in \mathfrak{g} \hookrightarrow \mathcal U(\mathfrak{g})$, the left-multiplication operator $L^\rhd_x$ acts as a derivation on elements in $\mathcal U(\mathfrak{g})$. Further below we will revisit these identities in the context of Gavrilov's $K$-map and special polynomials.
\end{remark}

For example, let $x,y,z \in \mathfrak{g}$, then, as these elements are primitive with respect to the unshuffle coproduct $\Delta_{\scriptscriptstyle{\shuffle}}$, we find
\allowdisplaybreaks
\begin{align*}
	x \cdot y 
		&= x \ast y - x \rhd y\\
	x \cdot y \cdot z 
		&= x \ast (y \cdot z) - x \rhd (y \cdot z)\\
		&= x \ast y \ast z - x \ast (y \rhd z) - x \rhd (y \ast z) + x \rhd (y \rhd z).
\end{align*}


\begin{remark}
\label{operators}\rm
From \eqref{product2} we immediately get
\begin{equation}
\label{product2a}
	(x_1 \ast \cdots \ast x_n) \rhd B = L^\rhd_{x_1} \cdots L^\rhd_{x_n} B.
\end{equation}
Hence, Grossman--Larson products of elements in $\mathfrak{g} \hookrightarrow \mathcal U(\mathfrak{g})$ transfer to compositions of left-multiplication maps. Therefore, any operator of the form $A \rhd -$, where $A \in \mathcal  U(\mathfrak{g})$, translates into a $L^\rhd$-polynomial. For example, for any $x,y\in\mathfrak g$,
$$
	(x\cdot y)\rhd B=(x \ast y - x \rhd y) \rhd B = (L^\rhd_x L^\rhd_y - L^\rhd_{x \rhd y}) \rhd B.
$$
In light of \eqref{product2a} and the recursion \eqref{RecursionL1}, we can extend the definition of the $L^\rhd$-operator to words, $x_1 \cdots x_n \in  \mathcal U(\mathfrak{g})$ by defining inductively
\begin{equation}
\label{RecursionL2}
	\hat{L}^\rhd_{x_1 \cdots x_n} := 
	L^\rhd_{x_1}  \hat{L}^\rhd_{x_2 \cdots x_n} - \sum_{i=2}^n \hat{L}^\rhd_{x_2 \cdots x_1 \rhd x_i \cdots x_n}.
\end{equation}
One checks that $\hat{L}^\rhd$ is an algebra morphism from $(\mathcal U(\mathfrak{g}),\ast)$ into $\text{End}(\mathcal U(\mathfrak{g}))$, that is,
\begin{equation}
\label{RecursionL3}
	\hat{L}^\rhd_{A \ast B} = \hat{L}^\rhd_{A}  \hat{L}^\rhd_{B},
\end{equation}
for $A,B \in \mathcal  U(\mathfrak{g})$. For example
\allowdisplaybreaks
\begin{align*}
	\hat{L}^\rhd_{x_1 \ast x_2} = \hat{L}^\rhd_{x_1 x_2 + x_1 \rhd x_2 } 
	&= \hat{L}^\rhd_{x_1 x_2} + L^\rhd_{x_1 \rhd x_2 } \\
	&= L^\rhd_{x_1}  L^\rhd_{x_2} - L^\rhd_{x_1 \rhd x_2} + L^\rhd_{x_1 \rhd x_2}\\
	&=L^\rhd_{x_1}  L^\rhd_{x_2}.
\end{align*}
Further below, we will see that these polynomials are closely related to Gavrilov's $K$-map.
\end{remark}


\subsection{Reminder on post-Lie Magnus expansion}
\label{sec:plm}

We consider now the free $D$-algebra generated by a set $A$, i.e., the universal enveloping algebra $\mathcal{F}^{pl}_A := \mathcal{U}\big(\mathcal{L}(\mathcal{T}^{pl}_A)\big)$ of the free post-Lie algebra $\big(\mathcal{L}(\mathcal{T}^{pl}_A), [.\,,\,.], \rhd \big)$, graded by the number of vertices of the forests\footnote{Here the magmatic product is not precised: in view of Remark \ref{lg}, it can be either the right Butcher product $\diamond$ or the left grafting $\curvearrowright$.}. Denote by $\widehat{\mathcal{U}\big(\mathcal{L}(\mathcal{T}^{pl}_A)\big)}$ its completion with respect to the grading. The unshuffle coproduct, $\Delta_\shuffle$, is naturally extended to the completion. The set $\mop{Prim}(\mathcal{F}^{pl}_A)$ consists in primitive elements, whereas $G(\mathcal{F}^{pl}_A)$ denotes the set  of group-like elements:
\allowdisplaybreaks
\begin{align}
		\mop{Prim}(\mathcal{F}^{pl}_A) 
		& := \big\{\alpha \in \widehat{\mathcal{U}(\mathcal{L}\big(\mathcal{T}^{pl}_A)\big)}\;|\, \Delta_\shuffle(\alpha) 
		=  \mathbf{1} \otimes \alpha + \alpha \otimes  \mathbf{1}\big\}
		=\widehat{\mathcal{L}(\mathcal{T}^{pl}_A)},\\
		G(\mathcal{F}^{pl}_A) 
		& := \big\{\alpha \in \widehat{\mathcal{U}\big(\mathcal{L}(\mathcal{T}^{pl}_A)\big)}\;|\, \Delta_\shuffle(\alpha) 
		= \alpha \otimes \alpha \big\}.&
\end{align}
	Both products on $\mathcal{U}\big(\mathcal{L}(\mathcal{T}^{pl}_A)\big)$ --the concatenation and the Grossman--Larson product-- can also be extended to products on the completion $\widehat{\mathcal{U}(\mathcal{L}\big(\mathcal{T}^{pl}_A)\big)}$. As a result, two different exponential functions can be defined on $\widehat{\mathcal{U}\big(\mathcal{L}(\mathcal{T}^{pl}_A)\big)}$, namely:
\allowdisplaybreaks
\begin{align*}
		\exp^{*}(y) 
		& = \sum_{n=0}^{\infty}{\frac{y^{*n}}{n!}} 
		= \mathbf{1} + y + \frac{1}{2} y \ast y + \frac{1}{6} y \ast y \ast y + \cdots, &\\
		\exp^{\textstyle\cdot}(y) 
		& = \sum_{n=0}^{\infty}{\frac{y^{{\textstyle\cdot} n}}{n!}} 
		= \mathbf{1} + y + \frac{1}{2} y {\textstyle{\cdot}} y + \frac{1}{6} y {\textstyle{\cdot}} y {\textstyle{\cdot}} y + \cdots.&
\end{align*}
On a manifold, the exponential $\exp^{*}(y)$ will be seen to represent the exact flow of a vector field, while $\exp^{\textstyle\cdot}(y)$ represents the flow along geodesics.

Both these exponential functions map $\mop{Prim}(\mathcal{F}^{pl}_A)$ bijectively onto $G(\mathcal{F}^{pl}_A)$. See \cite{ELM2015} for details. The post-Lie Magnus expansion $\chi$ is the bijective map from $\widehat{\mathcal{L}(\mathcal{T}^{pl}_A)}$ onto itself defined by the following relation between exponentials:
\begin{equation}
\label{exp=exp}
	\exp^{*}\big(\chi(y)\big) = \exp^{\textstyle\cdot}(y), 
\end{equation} 
namely:
\begin{equation}
\label{pl-Magnus} 
	\chi(y) = \log^{\ast}\big(\exp^{\textstyle\cdot}(y)\big).
\end{equation}  

The post-Lie Magnus expansion can be characterised by taking the derivation with respect to $t$ of $\exp^*(\chi(ty)) = \exp^{\textstyle\cdot}(ty)$. Recall the  $\text{dexp}$-formulas \cite{BCOR2009} for derivations of the $\exp$-map in a non-commutative setting
\allowdisplaybreaks
\begin{align}
	\frac{d}{dt} \exp(\Omega(t)) 
	&= \text{dexp}_{\Omega(t)}(\dot{\Omega}(t))\exp(\Omega(t)) \label{dexp1}\\ 
	&= \exp(\Omega(t))\text{dexp}_{-\Omega(t)}(\dot{\Omega}(t)). \label{dexp2}
\end{align}
Using the group-likeness of the two exponentials, $\exp^*(\chi(ty))$ and $\exp^\centerdot(ty)$, yields:
\allowdisplaybreaks
\begin{align}
	\frac{d}{dt}\exp^*(\chi(ty))
	&= \exp^\centerdot(ty) {\textstyle\cdot} y\\	
	&\stackrel{\eqref{productnew}}{=} \exp^\centerdot(ty) \ast \Big( S_*\big(\exp^\centerdot(ty)\big) \rhd y\Big)\\
	&\stackrel{\eqref{exp=exp}}{=} \exp^*(\chi(ty))  \ast \Big(S_*\big(\exp^*(\chi(ty))\big)  \rhd y\Big)\\
	&= \exp^*(\chi(ty)) \ast \Big(\exp^*(-\chi(ty)) \rhd y\Big), 	\nonumber
\end{align}
from which we deduce that 
$$
	\exp^*(-\chi(ty)) \ast \frac{d}{dt}\exp^*(\chi(ty))
	= \mathrm{dexp}^{*}_{-\chi(ty)}(\dot\chi(ty)) 
	= \exp^*(-\chi(ty)) \rhd y.
$$
Therefore, $\chi(ty)$ solves the Magnus-type differential equation
\begin{equation}
\label{plminv:eqn}
	\dot\chi(ty) = \mathrm{dexp}^{*-1}_{-\chi(ty)} \big(\exp^*\big(\!\!-\chi(ty)\big)\rhd y \big),
	\quad \chi(0) = 0.
\end{equation} 
The classical formula $\frac{d}{dt}\exp(-A(t)) = - \exp(-A(t)) \frac{d}{dt}\big(\exp(A(t))\big)\exp(-A(t))$ implies 
for the identity $\mathrm{dexp}^{*}_{-\chi(ty)}(\dot\chi(ty)) = \exp^*(-\chi(ty))  \rhd y$ that the function
\begin{equation}
\label{paralleltransp}
	\alpha(y,t):= \exp^*(-\chi(ty))  \rhd y
\end{equation} 
satisfies the differential equation
\allowdisplaybreaks
\begin{align}
	\lefteqn{\frac{d}{dt}\alpha(y,t) = \frac{d}{dt}\exp^*(-\chi(ty))  \rhd y } \\
	&= \Big(\frac{d}{dt}\exp^*(-\chi(ty))\Big)  \rhd x\\
	&= \Big(-\exp^*(-\chi(ty))*\big(\frac{d}{dt}\exp^*(\chi(ty))\big) * \exp^*(-\chi(ty))\Big)  \rhd y\\
	&= \Big(-\mathrm{dexp}^{*}_{-\chi(ty)}(\dot\chi(ty)) * \exp^*(-\chi(t y))\Big)  \rhd y\\
	&= -\mathrm{dexp}^{*}_{-\chi(ty)}(\dot\chi(ty)) \rhd \Big(\exp^*(-\chi(ty)) \rhd y\Big)\\
	&= - \alpha(y,t) \rhd \alpha(y,t).
\end{align}
Hence, $\alpha(y,t)$ satisfies the {\textit{post-Lie flow equation}} 
\begin{equation}
\label{eq:phi}
\left\{\begin{array}{l}
	\alpha(y,0)= y\\[0.2cm]
	\dfrac{d}{dt}\alpha(y,t)=-\alpha(y,t)\rhd \alpha(y,t).
\end{array}\right.
\end{equation}
We can therefore describe the Grossman--Larson exponential $\exp^*(\chi(ty))$ as a solution of a linear non-autonomous initial value problem
\begin{equation}
\label{theEq2}
	\frac{d}{dt}\exp^*(\chi(ty)) = \exp^*(\chi(ty)) * \alpha(y,t). 
\end{equation}

The inclusion of any element $y$ into the completed post-Lie algebra $\widehat{\mathcal{L}(\mathcal{T}^{pl}_A)}$ yields a unique injective morphism from the completed free magmatic algebra $\widehat{M_y}$ into it. We define now the map $\delta_{y}:\widehat{M_y}\to \widehat{M_y}$ to be the unique derivation, with respect to the magmatic product $\rhd$, such that $\delta_{y}y=y\rhd y$. For instance,
$$
	\delta_{y} (\delta_{y} y) = \delta_{y} (y \rhd y) = (y \rhd y) \rhd y + y \rhd  (y \rhd y).
$$ 
It is then clear that $e^{t\delta_{y}}$ is a one-parameter group of automorphisms for the product $\rhd$. This yields
$$
	\frac{d}{dt}e^{-t\delta_{y}}y
	=-e^{-t\delta_{y}}\delta_{y}(y)
	=-e^{-t\delta_{y}}(y \rhd y)
	=-(e^{-t\delta_{y}}y \rhd e^{-t\delta_{y}}y).
$$
This means that the function $t\mapsto e^{-t\delta_{y}}y$ solves as well the post-Lie flow equation \eqref{eq:phi} from which we deduce the intriguing identity  
\begin{equation}
\label{intriguing}
	\exp^*(-\chi(ty))  \rhd y = e^{-t\delta_{y}}y.
\end{equation}
The right-hand side of \eqref{intriguing} is therefore the purely magmatic expression of the map $\alpha$ introduced by A. V. Gavrilov in \cite{Gavrilov2006}, the left-hand side is its post-Lie reformulation.

\smallskip

The inverse of $\chi(ty)$, which we denote $\theta(ty)$, is obviously characterised by 
\begin{equation}
\label{exprelation1}
	\exp^*(ty) = \exp^{\textstyle\cdot}(\theta(ty)).
\end{equation}
From this we deduce in an analogous manner the differential equation
\begin{equation}
\label{plm:eqn}
	\dot\theta(ty) 
	= \mathrm{dexp}^{{\textstyle\cdot} -1}_{-\theta(ty)} \big(\exp^{\textstyle\cdot}\big(\theta(ty)\big) \rhd y \big), 
	\quad \theta(0) = 0.
\end{equation}
The first terms of the post-Lie Magnus expansion $\chi$ and its inverse $\theta$ are given by \cite[Appendix~A]{AEM21}
\begin{eqnarray}
	\chi(y)
	&=&y-\frac 12 y\rhd y+\frac{1}{12}y \rhd (y \rhd y) + \frac{1}{4}(y \rhd y) \rhd y + \frac{1}{12}[y \rhd y, y]+\cdots\nonumber\\
	&=&y-\frac 12 y\rhd y+\frac{1}{6}y \rhd (y \rhd y) + \frac{1}{6}(y \rhd y) \rhd y + \frac{1}{12}\llbracket y \rhd y, y\rrbracket+\cdots 
	\label{chi-order-three}\\[0.2cm]
	\theta(y)
	&=& y+\frac 12 y\rhd y+\frac{1}{6} y \rhd (y \rhd y) + \frac{1}{12} [y, y \rhd y]+\cdots\nonumber\\
	&=& y+\frac 12 y\rhd y+\frac{1}{12} y \rhd (y \rhd y) +  \frac{1}{12}(y \rhd y) \rhd y 
	+ \frac{1}{12} \llbracket y, y \rhd y\rrbracket+\cdots
	\label{theta-order-three}
\end{eqnarray}
Note that \eqref{chi-order-three} and \eqref{theta-order-three} follow from identity \eqref{twoLie}, which relates the Lie brackets 
$$
	(y \rhd y) \rhd y - y \rhd (y \rhd y) +[y \rhd y, y] = \llbracket y \rhd y, y\rrbracket.
$$
We remark that the post-Lie Magnus expansion $\chi$ and its inverse $\theta$ make sense in any complete graded post-Lie algebra.

\begin{remark}
Returning to the linear initial value problem \eqref{theEq2}, we see that $\exp^*(\chi(ty))=\exp^\ast\big({\Omega}(\alpha(y,t))\big),$ and therefore
$$
	\chi(ty)=\Omega(\alpha(y,t)),
$$
where the Magnus expansion 
$$
	{\Omega}\big(\alpha(y,t)\big)
	=\int_0^tds\ \sum_{n \ge 0}(-1)^n \frac{B_n}{n!} \ad^{*(n)}_{{\Omega}} \alpha(y,s).
$$
Here $B_n$ are the Bernoulli numbers, $B_0=1$, $B_1=1/2$, $B_2=1/6$, $B_3=0,\ldots$ and $\ad^{*}$ refers to the fact that the Grossman-Larson Lie bracket is used. Note that the  $(-1)^n$ factor on the righthand side affects only the first Bernoulli number $B_1$. Computing up to fourth order in $t$ we find
\begin{align*}
	\lefteqn{\Omega(\alpha(y,t))
	=ty - \frac{t^2}{2} y \rhd y 
					+ \frac{t^3}{6}\big( (y \rhd y) \rhd y 
					+ y \rhd (y \rhd y)\big) 
		 			+ \frac{t^3}{12} \llbracket y \rhd y , y\rrbracket}  \\
	&- \frac{t^4}{24}\big( ((y \rhd y) \rhd y) \rhd y 
					+ (y \rhd (y \rhd y) ) \rhd y 
					+ 2 (y \rhd y) \rhd (y \rhd y) \\
	&					
					+ y \rhd ((y \rhd y)  \rhd y)
					+ y \rhd (y \rhd (y \rhd y) )\big)	
	 + \frac{t^4}{24} \llbracket y, (y \rhd y) \rhd y 
					+ y \rhd (y \rhd y) \rrbracket 
	   				+ \cdots
		 \label{betamatch}
\end{align*}
This should be compared with the terms in \eqref{chi-order-three}.
\end{remark}

\subsection{Gavrilov's $K$-map}
\label{ssec:GavrilovK}

We recall here A.~V.~Gavrilov's $K$-map \cite{Gavrilov2006, Gavrilov2010}. We give an explicit formula for its inverse in terms of set partitions, and we show the link with noncommutative Bell polynomials \cite{EFLM2014,MunLun2011,SchimRid1996}. Finally, we recall the differential equation satisfied by the generating series \cite[Lemma 2]{Gavrilov2006}
$$
	K(\exp^{\textstyle\cdot}( ty ))=\sum_{n\ge 0} \frac{t^n}{n!} K(y^{{\textstyle\cdot} n}).
$$
As a consequence, its logarithm can be expressed in terms of the Magnus expansion (in its right-sided version) applied to Gavrilov's $\alpha$-map \cite[Section 5]{Gavrilov2006}. \ignore{The latter, as well as its close relative, the map denoted $\lambda$ \cite[Lemma 2]{Gavrilov2006}, is best understood in the planar rooted tree picture with the right-Butcher product $\diamond$ and left-grafting $\rhd$.}


\subsubsection{The $K$-map and the $D$-algebra structure}
\label{sssec:K-D}

Let $(M,\rhd)$ be any magmatic algebra, and let $T(M)$ be the tensor algebra over $M$, endowed with the $D$-algebra structure of Paragraph \ref{par:freeD}. For any element $y \in M$, let $\tau^{\rhd}_y$ be the linear endomorphism of $T(M)$ defined by $\tau^{\rhd}_y(w)=y \rhd w$ for any $w\in T(M)$. By \eqref{Dtensor1}, it is a derivation. The map 
$$
	K: T(M) \to T(M)
$$ 
is recursively defined by $K(1)=1$, $K(y)=y$ for any $y \in M$, and
\begin{equation}
\label{defK}
	K(y U)=y K(U)-K\circ \tau^{\rhd}_y(U)
\end{equation}
for any $y \in M$ and $U \in T(M)$. In particular,
$$	
	K(y_1 y_2)=y_1y_2 - y_1 \rhd y_2
$$
and
\begin{eqnarray*}
	\lefteqn{K(y_1 y_2 y_3)
	=y_1 K(y_2 y_3)-K \circ \tau^{\rhd}_{y_1}(y_2 y_3)}\\
	&=&y_1y_2 y_3 - y_1(y_2 \rhd y_3) - (y_1 \rhd y_2) y_3 - y_2 (y_1 \rhd y_3)
	 + y_2 \rhd(y_1 \rhd y_3) + (y_1 \rhd y_2) \rhd y_3.
\end{eqnarray*}
The map $K$ is clearly invertible, as $K(U)-U$ is a linear combination of terms of strictly smaller length than the length of $U \in T(M)$. The inverse $K^{-1}$ is uniquely determined by $K^{-1}(1)=1$, $K^{-1}(y)=y$ for any $y \in M$, and the recursive relation
\begin{equation}
\label{rec-Kinv}
	(L_y + \tau^{\rhd}_y) \circ K^{-1} = K^{-1}\circ L_y
\end{equation}
for any $y \in M$, where $L_y: T(M) \to T(M)$ is defined by $L_y(U):=y U$. Recall the Grossman--Larson product \eqref{glp} on $T(M)$ given in Definition \ref{def:GLprod}.

\begin{theorem}
\label{thm:isoK}
The $K$-map is a unital algebra isomorphism from $T(M)$ equipped with the Grossman--Larson product, $\big(T(M),\ast\big)$, onto the tensor algebra $\big(T(M),\cdot\big)$.
\end{theorem}

\begin{proof}
Recall that for $U,V \in T(M)$ we have $U \ast V:=U_{(1)}(U_{(2)} \rhd V).$ We prove 
$$
	K(U \ast V)=K(U)\cdot K(V)
$$ 
by induction on the length of the tensor $U \in T(M)$. The length zero case is trivial, the length one case is given by \eqref{defK}. Now let us compute, with $x\in M$, using the induction hypothesis:
\begin{eqnarray*}
	K\big((xU)\ast V\big)
	&=&K(x \ast U \ast V) - K\big((x \rhd U)\ast V)\\
	&=&K(x)\cdot K(U \ast V) - K(x \rhd U)\cdot K(V)\\
	&=&\big(xK(U)-K(x \rhd U)\big)\cdot K(V)\\
	&=&K(xU) \cdot K(V).
\end{eqnarray*}
\end{proof}

\subsubsection{An explicit formula for $K^{-1}$ in terms of set partitions}

An explicit formula for $K^{-1}(y_1\cdots y_n)$ is available in terms of set partitions of the strictly ordered set $[n]:=\{1,\ldots, n\}$: for any such partition $\pi$, let us denote its blocks by $B_1,\ldots, B_{|\pi|}$, where we have ordered them according to their maximum:
$$	
	\mop{max}B_1<\cdots <\mop{max} B_{|\pi|}.
$$
For any block $B$, say, of size $\ell$, define the element $y_B\in M$ by
$$
	y_B:=y_{b_1}\rhd\Big(y_{b_2}\rhd\big(\cdots 
	\rhd(y_{b_{\ell-1}}\rhd y_{b_\ell}) \cdots \big)\Big),
$$
where the elements $b_1<\cdots < b_\ell$ of $B$ are arranged in increasing order. For any partition $\pi$, let $(y_1\cdots y_n)^\pi\in T(M)$ be the element given by
$$
	(y_1\cdots y_n)^\pi:=y_{B_1}\cdots y_{B_{|\pi|}}.
$$

\begin{proposition}
\label{Kinverse}
$$
	K^{-1}(y_1\cdots y_n)=\sum_{\pi\smop{ set partition }\atop \smop{ of } \{1,\ldots, n\}}(y_1\cdots y_n)^\pi.
$$
\end{proposition}

\begin{proof}
We prove this result by induction on the length $n$. The cases $n=0$ and $n=1$ are trivial, the cases $n=2$ and $n=3$ read
\begin{align*}
	K^{-1}(y_1y_2) 
	&=y_1y_2 + y_1\rhd y_2,\\
	K^{-1}(y_1y_2y_3)
	&=y_1y_2y_3 
	+y_1 (y_2\rhd y_3)
	+y_2 (y_1\rhd y_3)
	+(y_1\rhd y_2)y_3 
	+ y_1\rhd (y_2 \rhd y_3).
\end{align*}
In the case of $n-2$, we have a sum of two terms which correspond to the two partitions of the set $\{1,2\}$, one with two blocks and one with two single blocks, respectively. The case $n=3$ includes all set partitions of order three. Supposing the result is true up to length $n$ we have, using \eqref{rec-Kinv}:
\begin{align*}
	\lefteqn{K^{-1}(y_0 y_1\cdots y_n)=y_0 K^{-1}(y_1\cdots y_n) +y_0\rhd K^{-1}(y_1\cdots y_n)}\\
	&= \sum_{\pi \smop{ set partition }\atop \smop{ of } \{1,\ldots, n\}} y_0(y_1\cdots y_n)^\pi
	+y_0\rhd (y_1\cdots y_n)^\pi\\
	&= \sum_{\rho \smop{ set partition }\atop \smop{ of } \{0,\ldots, n\},\,B_1=\{0\}}(y_0y_1\cdots y_n)^\rho\\
	& \hskip 5mm +\sum_{\pi\smop{ set partition }\atop \smop{ of } \{1,\ldots, n\}}
					\sum_{j=0}^{|\pi|}y_{B_1}\cdots y_{B_{j-1}}y_{B_j\sqcup\{0\}}y_{B_{j+1}}\cdots y_{B_{|\pi|}}\\
	&=\sum_{\rho\smop{ set partition }\atop \smop{ of } \{0,\ldots, n\},\,B_1=\{0\}}(y_0 y_1\cdots y_n)^\rho
	+ \sum_{\rho\smop{ set partition }\atop \smop{ of } \{0,\ldots, n\},\,B_1\neq \{0\}}(y_0 y_1\cdots y_n)^\rho\\
	&=\sum_{\rho\smop{ set partition }\atop \smop{ of } \{0,\ldots, n\}}(y_0y_1\cdots y_n)^\rho.
\end{align*}
\end{proof}

\begin{corollary}
\label{rec-K}
Gavrilov's $K$-map is recursively given by
\begin{align}
	K(y_1\cdots y_n)
	&\stackrel{\eqref{RecursionL1}}{=} y_1K(y_2 \cdots y_n) 
			- \sum_{i=2}^n K(y_2 \cdots (y_1 \rhd y_i) \cdots y_n) \label{rec-K-univ1} \\
	&=y_1\cdots y_n-\sum_{\pi \smop{ set partition } \atop \smop{ of } \{1,\ldots, n\},\, \pi\neq\hat 0}
	K\big((y_1\cdots y_n)^\pi\big), \label{rec-K-univ}
\end{align}
where $\hat 0$ stands for the unique partition of $\{1,\ldots,n\}$ with $n$ blocks.
\end{corollary}

\noindent In other words, from Theorem \ref{thm:isoK} and Proposition \ref{Kinverse}, we have 
$$
	y_1 \ast \cdots \ast y_n
	=y_1\cdots y_n 
		+\sum_{\pi \smop{ set partition } \atop \smop{ of } \{1,\ldots, n\},\, \pi\neq\hat 0}(y_1\cdots y_n)^\pi,
$$
which is in line with Remark \ref{rmk:GLproductrecursion}.

Extending the derivation $\tau^{\rhd}: M \to \text{Der}(T(M))$ to an algebra homomorphism $\hat{\tau}^{\rhd}: T(M) \to \text{End}_\mathbf{k}(T(M))$, and using  \eqref{RecursionL1} and \eqref{product2a} (we replace here $\hat{L}^\rhd$ with $\hat{\tau}^{\rhd}$), together with \eqref{rec-K-univ1} we can deduce a particular formula for the GL-product which already appeared in Gavrilov \cite[p.~1003]{Gavrilov2008}
$$
	A*B 
	=A_{(1)}(A_{(2)} \rhd B)
	=A_{(1)}\hat{\tau}^\rhd_{K(A_{(2)})}B.
$$


\begin{remark}\rm
Let us consider the terms $b_n:=K^{-1}(y^{n})$ for some $y\in M$. We can see these expressions as elements of the free magmatic algebra $M_y$ generated by $y$. From \eqref{rec-Kinv} we have
\begin{equation}
	b_n=(L_y+\delta_y)b_{n-1}
\end{equation}
Hence the $b_n$'s are the noncommutative Bell polynomials (\cite{SchimRid1996}, see also \cite{MunLun2011, EFLM2014}). 
\end{remark}

\subsubsection{Gavrilov's $\alpha$ and $\lambda$ maps, and the logarithm of $K(\exp^{\textstyle\cdot}(ty))$}
\label{par:Z}

The recursive expression \eqref{rec-K-univ} does not deliver easily a closed formula. Gavrilov tackled the problem from another angle in \cite{Gavrilov2006}, by solving a first-order linear differential equation verified by the generating function $t \mapsto K(\exp^{\textstyle\cdot}(ty)) \in T(M_y)[[t]]$. Namely, following reference \cite[Lemma 2]{Gavrilov2006} we have
\begin{equation}
\label{eq:Ke}
	\frac d{dt} K(\exp^{\textstyle\cdot}(ty))=K(\exp^{\textstyle\cdot}(ty))\cdot \alpha(y,t).
\end{equation}
Here $\alpha= \alpha(y,t)\in M_y[[t]]$ solves the initial value problem \eqref{eq:phi} \cite[Lemma 1]{Gavrilov2006}. The first terms of the series $\alpha(y,t)$ are given by
\begin{align}
\label{eq:alpha1}
	\alpha(y,t)
	&=e^{-t\delta_y }y\\
	&=y - t  y \rhd y + \frac{t^2}{2} \big( (y \rhd y) \rhd y + y \rhd (y \rhd y) \big) \\
	& \quad 
		- \frac{t^3}{3!} \big( ((y \rhd y) \rhd y) \rhd y 
		+ (y \rhd (y \rhd y) ) \rhd y 
		+ 2 (y \rhd y) \rhd (y \rhd y) \nonumber \\
	&\qquad	
		+ y \rhd ((y \rhd y) \rhd y) 
		+ y \rhd (y \rhd (y \rhd y))\big) +\cdots  \nonumber \\
	&=\Forest{[]}-t\,\Forest{[[]]}+\frac{t^2}{2!}\left(\Forest{[[][]]}+\Forest{[[[]]]}\right)
	-\frac{t^3}{3!}\left(\Forest{[[][][]]}+\Forest{[[[]][]]}
	+2\,\Forest{[[][[]]]}+\Forest{[[[][]]]}+\Forest{[[[[]]]]}\right)+\cdots \label{eq:alpha2}
\end{align}
In the last equality we have represented the monomials in the free magmatic algebra $(M_y, \rhd)$ as planar binary trees, the magmatic product $\rhd$ being the right Butcher product $\diamond$ here. In view of \eqref{diamond-rhd}, the derivation $\delta_y$ is given by the left grafting $y\curvearrowright -$. We encounter a planar version of the so-called Connes--Moscovici coefficients in front of the trees, which can be interpreted as the number of possible levelings of the corresponding planar binary tree obtained by inverse (right) Knuth rotation \cite{EMY2021}. The ordinary (non-planar) Connes--Moscovici coefficient of a rooted tree is obtained by summing up the coefficients of its planar representatives. 
Introducing one more generator $z$, the element
\begin{equation}
\label{eq:lambda}
	\lambda(ty,z):=e^{-t\delta_y }z\in T(M_{y,z})[[t]]
\end{equation}
satisfies the differential equation \cite[Lemma 3]{Gavrilov2006}
\begin{equation}
	\frac{d}{dt}\lambda(ty,z)
	=-\delta_y (e^{-t\delta_y}z)
	=-e^{-t\delta_y }(y\rhd z)
	=-\alpha(y,t)\rhd \lambda(ty,z).
\end{equation}

\begin{remark}
\label{rmk:alpha-lambda}
\rm
Going back to \eqref{intriguing}, we deduce from \eqref{eq:alpha2} and \eqref{eq:lambda}, using Proposition~\ref{prop:pre-glp},
\begin{eqnarray}
	\alpha(y,t) 
	&=& e^{-t\delta_y } y = \exp^{\ast} \big(-\chi(ty)\big) \rhd y,\\
	\lambda(ty,z) 
	&=& e^{-t\delta_y }z = \exp^{\ast} \big(-\chi(ty)\big) \rhd z. \label{lambda}
\end{eqnarray}
\end{remark}

The solution of \eqref{eq:Ke} is given by the exponential of the Magnus expansion \cite{M54} (in its right-sided version):
\begin{equation}
\label{Z-magnus-eq}
	K(\exp^{\textstyle\cdot}(ty))
	=\exp^{\textstyle\cdot}\Big(\int_0^t\dot{\Omega}[\alpha(y,-)](s)\, ds\Big),
\end{equation}
where $\dot{\Omega}=\dot{\Omega}[A]$ is implicitly defined for any series $A=A[t]$ in the indeterminate $t$ by
\begin{equation}
\label{eq:omega}
	\frac{d}{dt}\Omega[A](t)
	=\frac{\ad_{\Omega[A]}}{1-e^{-\ad_{\Omega[A]}}} A(t)
	=\sum_{n=0}^\infty\frac{(-1)^n B_n}{n!}\ad^n_{\Omega[A]} A(t).
\end{equation}
As before, the $B_n$ are the modified Bernoulli numbers. Integration of Eq.~(\ref{eq:omega}), leads to an infinite series for $\Omega[A]$, the first terms of which are
\[
	\Omega[A](t)
	=\int^t_0A(s_1)ds_1+\frac{1}{2}\int^t_0\left[\int^{s_1}_0A(s_2)ds_2, A(s_1)\right]ds_1+\cdots.
\]
Indeed, we have
$$
	\dot{\Omega}[s A]=\sum_{r\ge 1}s^r\widetilde A_r,
$$
where $s$ is an indeterminate, with $\widetilde A_1=A$, $\widetilde A_2=\frac12 A\pl A$ and the recursive procedure
\begin{align*}
	\widetilde A_r
	&=\sum_{m\ge 0}\frac{(-1)^mB_m}{m!}\sum_{r_1+\cdots+r_m=r-1}
	\widetilde A_{r_1}\pl\Big(\widetilde A_{r_2}\pl \big(\cdots 
	\pl (\widetilde A_{r_m}\pl A)\cdots\big)\Big).
\end{align*}
The binary operation $\pl$ is the pre-Lie product defined by
\begin{equation*}
	(A\pl B)(t):=\left[\int_0^t A(s)\, ds,\, B(t)\right].
\end{equation*}
For $A(t):=\alpha(y,t)=\sum_{\ell\ge 0}\alpha_\ell(y) t^\ell$ we therefore get
\begin{equation}
\label{eq:Z}
	Z(ty):=\log^{\textstyle\cdot} K(\exp^{\textstyle\cdot}(ty))=\sum_{n\ge 1}Z_n(y)t^n,
\end{equation}
where the coefficients $Z_n(y)$ are recursively given by $Z_1(y)=y$ and
\begin{equation}
\label{rec-Z}
	(n+1)Z_{n+1}(y)
	=\sum_{m\ge 0}\frac{(-1)^mB_m}{m!}
	\sum_{a_1+\cdots+a_m+\ell=n,\atop a_j\ge 1,\ell\ge 0} 
	\ad_{Z_{a_1}(y)}\circ\cdots\circ \ad_{Z_{a_m}(y)}\big(\alpha_\ell(y)\big).
\end{equation}


\noindent It turns out that the series $Z=Z(ty)$ is closely related to the post-Lie Magnus expansion $\chi=\chi(ty)$. Indeed, the definition of $\chi$, recalled in Paragraph \ref{sec:plm}, and Theorem \ref{thm:isoK}, which states that the $K$-map is a unital algebra isomorphism from $\big(T(M),\ast\big)$ to $\big(T(M),\cdot\big)$, together with identity \eqref{exp=exp} yield
\begin{eqnarray}
	Z(ty)	&=&\mop{log}^{\textstyle\cdot} K(\exp^{\textstyle\cdot} (ty))\nonumber\\
		&=&\mop{log}^{\textstyle\cdot} K\Big(\exp^{*} \big(\chi(ty)\big)\Big)\nonumber\\
		&=&\mop{log}^{\textstyle\cdot} \exp^{\textstyle\cdot}\big(K(\chi(ty))\big)\nonumber\\
		&=&K\big(\chi(ty)\big). \label{ZKchi}
\end{eqnarray}
Safely dropping the indeterminate $t$, the map $Z:= K\circ\chi$ makes sense as a map from $\widehat{\mop{Lie}(M)}$ into itself, where $M$ is any graded magmatic algebra and where the hat stands for completion.\\

\noindent From \eqref{chi-order-three} we therefore get
\begin{equation}
	Z(ty)=ty-\frac {t^2}2 y\rhd y+\frac{t^3}{6}y \rhd (y \rhd y) 
		+ \frac{t^3}{6}(y \rhd y) \rhd y + \frac{t^3}{12} [y \rhd y, y] +\cdots
%
\end{equation}
Here we use $K( \llbracket y \rhd y, y\rrbracket)= [y \rhd y, y]$. The series $Z$ in turn gives rise to Gavrilov's $\beta$-map \cite[Lemma 6]{Gavrilov2006}. We will return to this point in Section \ref{sect:framed-pl} below.

\section{Framed Lie algebras and Gavrilov's $\beta$ map}\label{sect:framed-pl}

\noindent We now recall Gavrilov's notion of framed Lie algebra \cite{Gavrilov2006}.

\begin{definition}\cite{Gavrilov2006}
\label{def:freeframedLie}
A framed Lie algebra is a triple $(\mathfrak l,\rhd,\lbc .\,,.\rbc)$ where $\rhd$ is any bilinear product on the vector space $\mathfrak l$, and where $\lbc.\,,.\rbc$ is a Lie bracket on $\mathfrak l$.
\end{definition}

No compatibility relations of any sort are requested between the magmatic product $\rhd$ and the Lie bracket, which we denote in bold, $\lbc.\,,.\rbc$, to stress the distinction with the two Lie brackets, $[.\,,.]$ and $\llbracket .\,,.\rrbracket$, of a post-Lie algebra. An obvious example is given by the Lie algebra of vector fields, $\mathfrak l=\mathcal {XM}$, on a manifold $\mathcal M$ endowed with an affine connection $\nabla$. In this case, the magmatic product is given by $X\rhd Y:=\nabla_XY$, and the natural Lie bracket is the usual Jacobi bracket.\\

The free framed Lie algebra generated by a single element $y$ is denoted by $\mathcal L_y$, and its completion with respect to the natural grading is denoted by $\widehat {\mathcal L}_y$. One also denotes by $\mathcal U(\mathcal L_y)$ (resp.~$\widehat{\mathcal U}(\mathcal L_y)$) the enveloping algebra of $(\mathcal L_y,\lbc.\,,.\rbc)$ (resp.~its completion). The canonical projection
$$
	p:T(\mathcal L_y) \to \hskip -10pt\to \mathcal U(\mathcal L_y)
$$
readily extends to the completion. Its restriction to the free Lie algebra $\mop{Lie}(\mathcal L_y)$ generated by $\mathcal L_y$ is a Lie algebra morphism from $\big(\mop{Lie}(\mathcal L_y),[.\,,.]\big)$ onto $(\mathcal L_y,\lbc.\,,.\rbc)$. Gavrilov showed the following
\begin{lemma}
\label{Gavrilov-main-lemma}\cite[Lemma 6]{Gavrilov2006}
There exists a unique series $\beta=\beta(y)$ in the completion $\widehat{\mathcal L}_y$ such that
\begin{equation}
\label{pkbeta}
	p\circ K(\exp^{\textstyle\cdot}(y))=\exp(\beta(y)) \in \widehat{\mathcal U}(\mathcal L_y).
\end{equation}
The series $\beta(ty)$ verifies $\beta(0)=0$ and
\begin{equation}
\label{eq-beta}
	\frac d{dt}\beta(ty)=\frac{ \ad_{\beta(ty)}}{1-e^{-\ad_{\beta(ty)}}}\alpha(y,t).
\end{equation}
\end{lemma}

\begin{proof}
The series $\beta(ty)=p\big(Z(ty)\big)$, with $Z$ defined in Paragraph \ref{par:Z}, is a solution to the problem. Uniqueness, as solution to an initial value problem, follows.
\end{proof}

\noindent From \eqref{ZKchi}, the series $\beta(ty)$ is explicitly given by
\begin{equation}
\label{beta-chi}
	\beta(ty)=p\circ K\circ\chi(ty).
\end{equation}
Safely dropping the indeterminate $t$, the map 
\begin{equation}
\label{betapKchi}
	\beta:=p\circ K\circ\chi\restr{\widehat{\mathcal L_y}}
\end{equation}
is a linear endomorphism of the completed framed Lie algebra $\widehat {\mathcal L}_y$. We therefore have from \eqref{chi-order-three}:
\begin{equation}
\label{beta-order-three-no-t}
	\beta(y)=y - \frac {1}2 y\rhd y + \frac{1}{6}y \rhd (y \rhd y) 
			+ \frac{1}{6}(y \rhd y) \rhd y 
			+ \frac{1}{12} \lbc y \rhd y, y\rbc
			+\cdots
\end{equation}

\begin{remark}\rm
\label{beta-g}
Note that \eqref{betapKchi} defines Gavrilov's $\beta$-map in the completion of any graded framed Lie algebra. It is obviously bijective due to the fact that $\beta$ is equal to the identity modulo higher degree terms.
\end{remark}

\section{Affine connections on manifolds and covariant derivation}
\label{sec:covderiv}

\subsection{Reminders on connections, torsion, curvature and Bianchi identities}

Let $E$ be a vector bundle on a smooth manifold $\mathcal M$, and let $\Gamma(E)$ be the $C^\infty(\mathcal M)$-module of its smooth sections. An affine connection on $E$ is a bilinear map
\begin{eqnarray*}
	\nabla:\mathcal{XM}\times \Gamma(E) & \longrightarrow 	&\Gamma(E)\\
							   (X,s) & \longmapsto	& \nabla_X s,
\end{eqnarray*}
subject to the relations
\begin{eqnarray}
	\nabla_{fX}s&=& f\nabla_X s,\\
	\nabla_X(fs)&=&(X.f)s+f\nabla_Xs \label{Leibniz-connection}
\end{eqnarray}
for any $f\in C^\infty(\mathcal M)$ and any $s\in\Gamma(E)$. For clarity, we write $\nabla^E$ if the vector bundle has to be made precise and we use the convenient notation
$$
	X\rhd s:=\nabla_X s.
$$
This includes the case of the trivial line bundle $L$, where $\Gamma(L)$ then coincides with $C^\infty(\mathcal M)$, and the natural connection is given by $X\rhd f:=X.f=\langle Df,X\rangle$. The Leibniz rule \eqref{Leibniz-connection} can therefore be rewritten as follows:
$$
	X\rhd(fs)=(X\rhd f)s+f(X\rhd s).
$$
Given a connection on two vector bundles, $E$ and $F$, connections on $E \otimes F$ and $\mop{Lin}(E,F)=E^*\otimes F$ are naturally given respectively by the Leibniz rules

\begin{eqnarray}
	X\rhd(s\otimes s')
	&=&(X\rhd s)\otimes s'+s\otimes(X\rhd s'),\\
	(X\rhd \varphi)(s)
	&=&X\rhd\varphi(s)-\varphi(X\rhd s).
\end{eqnarray}
The curvature of an affine connection is given by
\begin{equation}
	r(X,Y)s:=X\rhd (Y\rhd s)-Y\rhd(X\rhd s)-\lbc X,Y \rbc \rhd s
\end{equation}
and is also denoted by $R(X,Y,s)$. It is skew symmetric in $(X,Y)$ and $C^\infty$-linear with respect to each of the three arguments. In the case $E=T\mathcal M$, the torsion is given by
\begin{equation}
	t(X,Y):=X\rhd Y-Y\rhd X -\lbc X,Y \rbc.
\end{equation}
This is skew-symmetric, and $C^\infty$-linear with respect to both arguments. The two Bianchi identities are given by
\begin{eqnarray}
\oint_{XYZ}R(X,Y,Z)
	&=&\oint_{XYZ}(X\rhd t)(Y,Z) -\oint_{XYZ}t\big(X,\,t(Y,Z)\big) \label{bianchi-one}\\
	\oint_{XYZ}(X\rhd R)(Y,Z,W)&=&\oint_{XYZ}R\big(X,\,t(Y,Z),W)\label{bianchi-two}
\end{eqnarray}
for any $X,Y,Z,W\in\mathcal{XM}$. Here the symbol $\oint$ stands for summing over circular permutations of the three arguments. For a detailed account, see e.g.~\cite{KN1963}.

\subsection{Higher-order covariant derivatives}
\label{tc-postLie}

We keep the notations from the previous paragraph and consider the post-Lie algebra associated to a manifold with connection. Let $\mathcal R:=C^\infty(\mathcal M)$. The $\mathcal R$-module $\mathcal{XM}=\mop{Der}(\mathcal R)$ is denoted by $\mathcal V$. We adopt the notations $X \rhd Y:= \nabla_X Y$, for $X,Y \in \mathcal V$, and $X \rhd f:=X.f$, for $f \in \mathcal R$. Let $\mathcal{DM}$ be the algebra of differential operators on $\mathcal M$, which is the subalgebra of linear operators on $C^\infty(\mathcal M)$ generated by the vector fields and the multiplication operators $\mu_f:g\mapsto fg$.\\

Let $A(T_m \mathcal M)$ and $\mop{Lie}(T_m \mathcal M)$ be the free unital associative algebra (i.e.~the tensor algebra) respectively the free Lie algebra both defined over the tangent space at any point $m \in \mathcal M$. Each of these free algebras put together form respectively the free unital algebra bundle $A_{\mathcal M}$ and the free Lie algebra bundle $\mop{Lie}_{\mathcal M}$. The free $\mathcal R$-associative unital algebra $\mathcal A=T_{\mathcal R}(\mathcal V)$ on $\mathcal V$ is the $\mathcal R$-module of sections of $A_{\mathcal M}$. We clearly have
\begin{equation}
	\mathcal A=T_{\mathbb R}(\mathcal V)/\mathcal C,
\end{equation}
where $\mathcal C$ is the two-sided ideal generated by the elements $fX\cdot Y - X \cdot fY$ with $f\in C^\infty(\mathcal M)$ and $X,Y \in \mathcal V$. We denote by $\pi$ the natural projection from $T_{\mathbb R}(\mathcal V)$ onto $\mathcal A$ (let us recall that $T_{\mathbb R}(\mathcal V)$ is the free $D$-algebra generated by the magmatic algebra $(\mathcal V,\rhd)$). Similarly, the $\mathcal R$-module $\mathfrak{g}$ of sections of $\mop{Lie}_{\mathcal M}$ is the free $\mathcal R$-Lie algebra $\mop{Lie}_{\mathcal R}(\mathcal V)$ on the vector fields and we have
\begin{equation}
	\mathfrak g=\widetilde{\mathfrak g}/(\mathcal C\cap\widetilde{\mathfrak g}),
\end{equation}
where $\widetilde{\mathfrak g}=\mop{Lie}_{\mathbb R}(\mathcal V)$ is the free post-Lie algebra generated by $(\mathcal V,\rhd)$. The tautological action of $\mathcal V$ by derivations on $\mathcal R$ is extended to $\widetilde{\mathfrak g}$ as follows:
$$
	[X,Y] \rhd  f:= X \rhd (Y \rhd f) 
						- (X \rhd Y) \rhd f 
						- Y \rhd (X \rhd f) 
						- (Y \rhd X) \rhd f,
$$ 
and similarly, for $X\in \mathcal V$ and $U\in T_{\mathbb R}(\mathcal V)$,
$$
	(X\cdot U)\rhd f:=X\rhd(U\rhd f)-(X\rhd U)\rhd f.
$$
These rules, which define the action recursively with respect to the degree, are adapted from the second post-Lie axiom and the second $D$-algebra axiom, respectively. 

\begin{theorem}\label{higher-cov}
The map $\rhd:T_{\mathbb R}(\mathcal V)\times \mathcal R\to \mathcal R$ defined above factorizes into a map
$$
	\rhd:\mathcal A\times \mathcal R\to \mathcal R.
$$
In other words, $\mop{Id}_{\mathcal V}$ extends to a surjective $\mathcal R$-linear morphism 
$$
	\rho: \mathcal A\to\hskip -10pt \to\mathcal{DM}.
$$
It restricts to
$$
	\rho: \mathfrak{g} \to\hskip -10pt \to \mathcal V.
$$
\end{theorem}

\begin{proof}
It suffices to prove that
$$
	\rho(x_1\cdots x_n)=(x_1\cdots x_n)\rhd y
$$
is $C^\infty(\mathcal M)$-linear in each argument $x_i\in\mathcal{XM}$. This is obvious for $n=1$, and proven by induction for $n\ge 2$ using \eqref{Dtensor1} and \eqref{Dtensor2}:
\begin{align*}
	\lefteqn{(f_1x_1\cdots f_nx_n)\rhd y=f_1x_1\rhd\big((f_2x_2\cdots f_nx_n)\rhd y\big)
		- \big(f_1x_1\rhd(f_2x_2\cdots f_nx_n)\big)\rhd y} \hspace{1.5cm}\\
	&= f_1x_1\rhd\big(f_2\cdots f_n(x_2\cdots x_n)\rhd y\big)
		- \big(f_1x_1\rhd(f_2x_2\cdots f_nx_n)\big)\rhd y\\
	&= f_1\cdots f_nx_1\rhd\big(x_2\cdots x_n\rhd y\big)+\big(f_1x_1\rhd(f_2\cdots f_n)\big)(x_2\cdots x_n\rhd y)\\
	& \qquad -f_1\sum_{i=2}^n\Big(f_2x_2\cdots \big((x_1\rhd f_i)x_i+f_i(x_1\rhd x_i)\big)\cdots x_n\Big)\rhd y\\
	&= f_1\cdots f_n\Big(x_1\rhd\big((x_2\cdots x_n)\rhd y\big)-\big(x_1\rhd(x_2\cdots x_n)\big)\rhd y\Big).
\end{align*}
\end{proof}

\begin{remark}\label{hcd-extend}\rm
The previous formalism can be generalized to any vector bundle $E$ endowed with a connection $\nabla^E$, replacing $\mathcal{XM}$ by the $C^\infty(\mathcal M)$-module $\Gamma(E)$ of smooth sections of $E$. The higher-order covariant derivatives are also recursively defined by \eqref{Dtensor2}, and Theorem \ref{higher-cov} also holds for the action $\rhd:T_{\mathbb R}(V)\times \Gamma(E)\to\Gamma(E)$.
\end{remark}

We note that several notations are used in the literature for higher covariant derivatives, namely \cite{Gavrilov2008}
\begin{eqnarray*}
	\rho(x_1\cdots x_n)y
	&=&(x_1\cdots x_n)\rhd y\\
	&=&\nabla^n_{x_1\cdots x_n}y\\
	&=&(x_1\cdots x_n)\nabla^n y\\
	&=&\nabla^ny(x_n;\cdots;x_1).
\end{eqnarray*}

\begin{proposition}
\label{prop:g-post-lie}
$\mathfrak g$ (resp.~$\mathcal A$) is a post-Lie algebra (resp.~a $D$-algebra), and $\mathfrak g=\mathfrak d(\mathcal A)$.
\end{proposition}

\begin{proof}
It is sufficient to show that the Lie ideal $\mathcal C':=\mathcal C\cap \widetilde{\mathfrak g}$, generated by the elements $[fA,B]-[A,fB]$, $f \in \mathcal R$, $A,B \in \widetilde{\mathfrak g}$, is also an ideal for the product $\rhd$ extended to $\widetilde{\mathfrak g}$. Let $U,A,B\in\mathfrak g$ and $f\in \mathcal R$. From the straightforward computation
\begin{eqnarray*}
	U\rhd([fA,B]-[A,fB])
	&=&[(U\rhd f)A,B]+[f(U\rhd A),B]+[fA,U\rhd B]\\
	&  &-[U\rhd A,fB]-[A,(U\rhd f)B]-[A,f(U\rhd B)]
\end{eqnarray*}
we get $U\rhd\mathcal C'\subseteq\mathcal C'$ for any $U\in\mathfrak g$. From
\begin{align*}
	\lefteqn{([fA,B]-[A,fB])\rhd U=fA\rhd(B\rhd U)-(fA\rhd B)\rhd U - B\rhd(fA\rhd U)}\\
	& \quad  +(B\rhd fA)\rhd U - A\rhd(fB\rhd U)+(A\rhd fB)\rhd U\\
	& \quad +fB\rhd(A\rhd U) -(fB\rhd A)\rhd U\\
	&=-(A\rhd f)B\rhd U+(A\rhd f)B\rhd U-(B\rhd f)A\rhd U+(B\rhd f)A\rhd U\\
	&=0
\end{align*}
we get $\mathcal C'\rhd U\subseteq\mathcal C'$ for any $U\in\mathfrak g$. The proof of the fact that $\mathcal A$ is a $D$-algebra is similar, and the last assertion is clear.
\end{proof}

\noindent Let us recall two more Lie brackets at hand:
\begin{itemize}
\item
The usual Jacobi bracket $\lbc .\,,. \rbc$ on $ \mathcal V$, defined by 
$$
	\lbc X,Y \rbc\  \rhd f 
	= X \rhd (Y \rhd f) - Y \rhd (X \rhd f)
$$
for any $X,Y\in\mathfrak g$ and any $f \in \mathcal R$,
\item
The Grossman--Larson bracket $\llbracket  .\,,.  \rrbracket$ on $\mathfrak{g}$, which satisfies for any $Z \in \mathfrak{g}$ 
$$
	\llbracket X,Y \rrbracket \rhd Z 
	= X \rhd (Y \rhd Z) - Y \rhd (X \rhd Z).
$$
\end{itemize}
From the post-Lie algebra identity $\llbracket X,Y \rrbracket = [ X,Y ] + X \rhd Y - Y \rhd X$, for $X,Y \in \mathfrak{g}$ we get
\begin{equation}
\label{rhoLie}
	\rho\llbracket X,Y \rrbracket 
	= \rho(X) \circ \rho(Y) - \rho(Y) \circ \rho(X)
	= \lbc\rho(X),\rho(Y)\rbc. 
\end{equation}
Hence, $(\mop{Lie}_{\mathcal M},\,\llbracket.\,,.\rrbracket)$ is a Lie algebroid on $\mathcal M$ with anchor map $\rho$. We refer the reader to \cite{munthe2020invariant} which discusses in detail the fact that Post-Lie algebroids are action algebroids.\\

The canonical projection $p:T_{\mathbb R}(\mathcal V) \to \hskip -10pt\to \mathcal U( \mathcal V)$ restricts to $p:\widetilde{\mathfrak g}\to\hskip -10pt \to \mathcal V$. The two following diagrams commute (see also \cite[Lemma 2]{Gavrilov2008}): 
$$
\xymatrix{(T_{\mathbb R}(\mathcal V),*)\ar[rr]^K_\sim\ar@{>>}[d]^\pi &&(T_{\mathbb R}(\mathcal V),\cdot)\ar@{>>}[d]^p\\
(\mathcal A,*)\ar@{>>}[drr]_\rho &&\mathcal U(\mathcal V)\ar@{>>}[d]^{\widetilde\pi}\\
&& \mathcal{DM}
}
\hskip 12mm
\xymatrix{(\widetilde{\mathfrak g},\llbracket .\,,.\rrbracket)\ar[rr]^K_\sim\ar@{>>}[d]^\pi &&(\widetilde{\mathfrak g},[.\,,.])\ar@{>>}[dd]^p\\
(\mathfrak g,\llbracket .\,,.\rrbracket)\ar@{>>}[drr]_\rho &&\\
&& \mathcal V
}
$$
Here the map $\widetilde\pi$ stands for the natural projector from the universal enveloping algebra $\mathcal U(\mathcal V,\lbc.\,,.\rbc)$ onto $\mathcal{DM}$, and, in view of Theorem \ref{thm:isoK}, all arrows are algebra (resp.~Lie algebra) morphisms. 

\begin{remark}
\label{alt-beta}\rm
Gavrilov's $\beta$-map is a bijection from $\overline{\mathcal V}$ into itself, where $\overline{\mathcal V}$ is the complete filtered framed Lie algebra obtained from $\mathcal V=\mathcal{XM}$ by extending the scalars from real numbers $\mathbb R$ to power series without constant terms, in one or several indeterminates, e.g.~$\overline{\mathcal V}=t\mathcal V[[t]]$ or $\overline{\mathcal V}=t\mathcal V[[t,s]]+s\mathcal V[[t,s]]$. On the other hand, both graded post-Lie algebras, $\widetilde{\mathfrak g}$ and $\mathfrak g$, come with their own post-Lie Magnus expansions, $\widetilde\chi$ respectively $\chi$. Using the natural extensions of $\widetilde\chi$, $\chi$ and the projection $\pi$ to the associated completed versions of $\widetilde{\mathfrak g}$ and $\mathfrak g$, we have
\begin{equation} 
\pi\circ\widetilde\chi=\chi\circ\pi.
\end{equation}
In view of \eqref{beta-chi} we therefore have
\begin{equation}
\label{beta-chi-two}
	\beta = p\circ K\circ \widetilde\chi\restr{\overline{\mathcal V}}
		= \rho\circ\pi\circ\widetilde\chi\restr{\overline{\mathcal V}}
		= \rho\circ\chi\circ\pi\restr{\overline{\mathcal V}},
\end{equation}
hence
\begin{equation}
\label{betachirho}
	\beta=\rho\circ\chi\restr{\overline{\mathcal V}}\ .
\end{equation}
The situation can be summarized by the following diagram:
$$
\xymatrix{
\overline{\widetilde{\mathfrak g}}\ar[rr]^{\widetilde\chi}_\sim\ar@{>>}[d]^\pi &&(\overline{\widetilde{\mathfrak g}},\llbracket .\,,.\rrbracket)\ar[rr]^K_\sim\ar@{>>}[d]^\pi &&(\overline{\widetilde{\mathfrak g}},[.\,,.])\ar@{>>}[dd]^p\\
\overline{\mathfrak g}\ar[rr]^\chi_\sim &&(\overline{\mathfrak g},\llbracket .\,,.\rrbracket)\ar@{>>}[drr]_\rho &&\\
\overline{\mathcal V}\ar@{^{(}->}[u]\ar[rrrr]^\beta_\sim &&&& \overline{\mathcal V}
}
$$
\end{remark}

\subsection{Differential operators}
\label{par:diff}

It is easily seen that we can identify the algebra $\mathcal{DM}$ of differential operators with $\mathcal U(\mathcal V)/\mathcal I$, where $\mathcal I$ is the two-sided ideal generated by the elements $X(fv)-(fX)v-(X\rhd f)v$ for $f\in C^\infty(\mathcal M)$, $X\in \mathcal V$ and $v\in \mathcal U(\mathcal V)$.\\

Now we develop Example \ref{example2} outlined above. Supposing that the connection is flat with constant torsion, the $D$-algebra structure on $\mathcal{DM}$ will immediately arise in view of the following result

\begin{lemma}
\label{lem:dalgExample2}
When the affine connection $\nabla$ on $\mathcal M$ is flat with constant torsion, $(\mathcal V,\rhd,\lbc.\,,.\rbc)$ is a post-Lie algebra, and both projections
$T(\mathcal {XM})\mopl{$\longrightarrow\hskip -6mm\longrightarrow$}^{p} \mathcal U(\mathcal {XM})\mopl{$\longrightarrow\hskip -6mm\longrightarrow$}^{\widetilde\pi}\mathcal {DM}$
are $D$-algebra morphisms.
\end{lemma}

\begin{proof}
It suffices to show that the ideal $\mathcal I$ is also a two-sided ideal with respect to the magmatic product $\rhd$. The proof, left to the reader, uses the cocommutativity of the coproduct on $\mathcal U(\mathcal{XM})$, and proceeds similarly to the one of Proposition \ref{prop:quotient}. 
\end{proof}

\section{Special polynomials}
\label{sect:special}

We now study Gavrilov's special polynomials \cite{Gavrilov2006} from the post-Lie viewpoint. 

\subsection{Torsion and curvature revisited in the post-Lie framework}
\label{tc-postLie-bis}

We use the notations introduced at the beginning of Section \ref{sec:covderiv}.

\begin{definition}
\label{torsion-curvature-elts}
The torsion of two elements $a,b \in \mathcal V$ is defined by
$$
	t(a.b)
	:=a \rhd b - b \rhd a - \lbc a,b\rbc. 
$$
The curvature of three elements $a,b,c\in \mathcal V$ is defined by
$$
	r(a.b)(c)
	:= a \rhd (b \rhd c)  
		- b \rhd (a \rhd c) 
			- \lbc a,b\rbc \rhd c. 
$$
\end{definition}

This is sometimes denoted $R(a.b.c)$. One can show that $t \in {\mop{Lin}}_\mathcal R(\mathcal V \otimes_{\mathcal R} \mathcal V,\mathcal V)$ (a tensor of type (2,1)) and that $R \in {\mop{Lin}}_{\mathcal R}(\mathcal V \otimes_{\mathcal R} \mathcal V \otimes_{\mathcal R} \mathcal V,\mathcal V)$  (a tensor of type (3,1)). We can rewrite the curvature in our post-Lie framework as follows:
\begin{equation}
\label{rs1}
	r(a.b)(c) = s(a.b) \rhd c.
\end{equation}
On the righthand side of \eqref{rs1}, we identify a new element.

\begin{definition}
\label{def:curvature element}
The \textsl{curvature element} $s(a.b) \in \mathfrak g$ is defined by
\begin{equation}
\label{rs2}
	s(a.b)
	:= \llbracket a,b \rrbracket - \lbc a,b\rbc 
	= [a,b] +  t(a.b).
\end{equation}
\end{definition}

In turn, $s(a.b)$ permits to express the torsion in terms of the three Lie brackets involved:
\begin{equation}
\label{rs3}
	t(a.b)= \llbracket a,b \rrbracket -[a,b]-\lbc a,b\rbc.
\end{equation}

\subsection{The ideals $\mathcal J$ and $\mathcal K$}
\label{ssec:JandKideals}

\begin{proposition}
\label{prop:kernel}
Let $\mathcal J:=\mop{Ker}\rho \subset \mathfrak{g}$, and let $\tilde{\mathcal J}$ be the ideal of $\mathfrak{g}$, for the Grossman--Larson bracket (GL-bracket), generated by the curvature elements $s(a.b)$, $a,b \in \mathcal V$. Then
$$
	\tilde{\mathcal J} = \mathcal J.
$$
We also have the decomposition 
\begin{equation}
\label{gVJ}
	\mathfrak{g} = \mathcal V \oplus \mathcal J.
\end{equation}
\end{proposition}

\begin{proof}
The direct sum decomposition \eqref{gVJ} is immediate in view of $\mathcal V=\rho(\mathfrak{g})$ and $\mathcal J=(I-\rho)(\mathfrak{g})$. From \eqref{rhoLie} and \eqref{rs2}, we immediately get the inclusion $\tilde{\mathcal J} \subseteq \mathcal J$, as well as the fact that $\mathcal J$ is an ideal for the GL-bracket.

\medskip

Conversely, we use the grading $\mathfrak{g}=\oplus_{n \ge 0} \mathfrak{g}_n$ by the length of the iterated brackets. Looking at the definition of the GL-bracket, it is easy to show that $\mathfrak{g}^{(m)}:=\oplus_{i=1}^m \mathfrak{g}_i$ is the $\mathcal R$-linear span of iterated GL-brackets of length smaller or equal to $m$. By using Jacobi identity as many times as necessary, any such iterated bracket can be rewritten as a sum 
$$
	\sum_i \llbracket a_i,v_i \rrbracket
$$
with $a_i \in \mathcal V$ and $v_i \in \mathfrak{g}^{(n-1)}$. Suppose now that any element of $\mathfrak{g}^{(n-1)} \cap \mathcal J$ is in $\mathfrak{g}^{(n-1)} \cap \tilde{\mathcal J}$. This is trivial for $n-1=1$ and clear for $n-1=2$ in view of \eqref{rs2}. Considering any element $u=\sum_i \llbracket a_i,v_i \rrbracket \in \mathfrak{g}^{(n)}$ , we have 
\begin{align}
	u - \rho(u) 
	&= \sum_i \Big( \llbracket a_i,v_i \rrbracket - \lbc a_i,\rho(v_i)\rbc  \Big)\\
	&= \sum_i \llbracket a_i,v_i - \rho(v_i) \rrbracket 
		+  \sum_i \Big( \llbracket a_i,\rho(v_i) \rrbracket - \lbc a_i,\rho(v_i)\rbc  \Big)\\
	&= \sum_i \llbracket a_i,v_i - \rho(v_i) \rrbracket 
		+ \sum_i s\big(a_i.\rho(v_i)\big), 	
\end{align}
which proves $\mathcal J \subseteq \tilde{\mathcal J}$ by induction on $n$.
\end{proof}

\begin{proposition}
\label{prop:identities}
For any $u\in\mathfrak g$ and $b,c,d\in \mathcal V$ we have
\begin{eqnarray*}
	\big(u\rhd r(b.c)\big)(d)
	&=&\llbracket u,\,s(b.c)\rrbracket\rhd d,\\
	\big((u\rhd r)(b.c)\big)(d)
	&=&\Big(\llbracket u,\,s(b.c)\rrbracket-s\big((u\rhd b).c\big)-s\big(b.(u\rhd c)\big)\Big)\rhd d.
\end{eqnarray*}
\end{proposition}

\begin{proof}
From Leibniz rule we have
\begin{eqnarray}
	u\rhd\big(s(b.c)\rhd d\big)
	&=&u\rhd\big(r(b.c)(d)\big)\nonumber\\
	&=&\big(u\rhd r(b.c)\big)(d)+r(b.c)(u\rhd d)\nonumber\\
	&=&\big(u\rhd r(b.c)\big)(d)+s(b.c)\rhd(u\rhd d).\label{rs-one}
\end{eqnarray}
Now, we have
\begin{equation}
\label{rs-two}
	u\rhd\big(s(b.c)\rhd d\big)-s(b.c)\rhd(u\rhd d)
	=\llbracket u, s(b.c)\rrbracket \rhd d.
\end{equation}
From \eqref{rs-one} and \eqref{rs-two} we get
$$
	\big(u\rhd r(b.c)\big)(d)=\llbracket u,s(b.c)\rrbracket\rhd d,
$$
which proves the first assertion. The second one comes from the first together with the Leibniz rule
\begin{eqnarray*}
	\big(u\rhd r)(b.c)\big)(d)
	&=&\big(u\rhd r(b.c)\big)(d)-r\big((u\rhd b).c\big)(d)-r(b.(u\rhd c)\big)(d)\\
	&=&\big(u\rhd r(b.c)\big)(d)-s\big((u\rhd b).c\big)\rhd d-s(b.(u\rhd c)\big)\rhd d.
\end{eqnarray*}
\end{proof}

\noindent Now let us consider the map
\begin{eqnarray*}
	\Phi:\mathfrak g		&\longrightarrow	& \mop{End}_{\mathbb R}(\mathcal V)\\
				u 	&\longmapsto		& u\rhd -,
\end{eqnarray*}
and introduce $\mathcal K:=\mop{Ker}\Phi$. Using earlier notation, $\Phi(u)=L_u^\rhd$.

\begin{proposition}
The restriction of $\Phi$ to $\mathcal J$ takes its values into $\mop{End}_{\mathcal R}(\mathcal V)$. Moreover, $\mathcal K$ is an ideal for the Grossman--Larson bracket, and we have the strict inclusions
$$
	\{0\}\subsetneq\mathcal K\subsetneq \mathcal J.
$$
\end{proposition}

\begin{proof}
The first assertion is immediate from the Leibniz rule. From $\llbracket u,v\rrbracket\rhd a=u\rhd(v\rhd a)-v\rhd(u\rhd a)$ for any $u,v\in \mathfrak g$ and $a\in \mathcal V$ we get $\Phi(\llbracket u,v\rrbracket)=[\Phi(u),\,\Phi(v)]$ (bracket of operators on $\mathcal V$), hence $\mathcal K$ is an ideal for the Grossman--Larson bracket.

\smallskip

Now, for any $f\in \mathcal R$, $a\in \mathcal V$ and $u\in\mathcal K$ we have $u\rhd a=u\rhd fa=0$, therefore $u\rhd f=0$ by Leibniz rule, hence $\mathcal K\subset\mathcal J$. Moreover, for any $a,b\in \mathcal V$, the curvature element $s(a.b)$ belongs to $\mathcal J$, but has no reason to belong to $\mathcal K$ unless the connection is flat. Finally, from the second identity of Proposition \ref{prop:identities} and from the differential Bianchi identity \eqref{bianchi-two}, the expression
$$
	\llbracket a,s(b.c) \rrbracket+\llbracket b,s(c.a) \rrbracket
	+\llbracket c,s(a.b) \rrbracket 
	-s\big(a.t(b.c)\big)-s\big(b.t(c.a)\big)-s\big(c.t(a.b)\big)
$$
$$
	+s\big(a.(b\rhd c-c\rhd b)\big)+s\big(b.(c\rhd a-a\rhd c)\big)+s\big(c.(a\rhd b-b\rhd a)\big)
$$
defines a nontrivial element of $\mathcal K$, which can also be rewritten as
$$
	\llbracket a,s(b.c) \rrbracket+\llbracket b,s(c.a) \rrbracket
	+\llbracket c,s(a.b) \rrbracket +s\big(a.[b,c]\big)+s\big(b.[c,a]\big)+s\big(c.[a,b]\big) \in \mathcal K.
$$
\end{proof}

\subsection{Lie monomials}
\label{ssec:Liepoly}

We denote by $T_s^r(\mathcal V)$ the space of tensors of type $(r,s)$, namely 
$$
	T^r_s(\mathcal V):=\underbrace{\mathcal V \otimes _\mathcal{R}  \cdots \otimes _\mathcal{R} \mathcal V}_r
	\otimes _\mathcal{R} \underbrace{{\mathcal V}^* \otimes _\mathcal{R} \cdots \otimes _\mathcal{R} {\mathcal V}^*}_s,
$$
such that $T^1_0(\mathcal V):=\mathcal V$ and $T^0_0(\mathcal V):=\mathcal R$.

\begin{definition}
A Lie monomial of degree $n$ is a $\mathcal R$-linear map $\alpha:T^n_0(\mathcal V)\to\mathfrak g_n$ defined by an  iteration of Lie brackets. In particular, it is a tensor of type $(n,n)$.
\end{definition}

\noindent As an example, consider
$$
	\alpha(a.b.c.d.e):=\Big[[a,b],\, \big[c,[d,e]\big]\Big],
$$
which defines a Lie monomial of degree $5$. The following statement is straightforward and left to the reader:

\begin{proposition}
\label{lem:LieMon}
A degree $n$ Lie monomial $w \mapsto \alpha(w) \in \mathfrak{g}_n$ defines three tensors
\begin{align}
	\alpha &\in T^n_n(\mathcal V), \label{i}\\
	-t_\alpha&:=\rho \circ \alpha \in T^1_n(\mathcal V), \label{ii}\\
	R_\alpha&:=x_1\cdots x_{n+1}\mapsto \Big(\big((I-\rho)\circ \alpha\big)(x_1\cdots x_n)\Big)\rhd x_{n+1} 
	\in T^1_{n+1}(\mathcal V). \label{iii}
\end{align}
\end{proposition}

Note the minus sign in \eqref{ii}, so that the definition matches torsion and curvature for $\alpha(a.b):=[a,b]$. For later use, we also define
\begin{equation}
	s_\alpha:=(I-\rho)\circ\alpha,
\end{equation}
so that
\begin{equation}
	R_\alpha(x_1\cdots x_{n+1})=s_\alpha(x_1\cdots x_n)\rhd x_{n+1}.
\end{equation}
Let us give the $\mathcal V \oplus \mathcal J$ decomposition of Lie monomials of low degrees:
\begin{itemize}
	\item Degree one: $\alpha=\mop{Id}_{\mathcal V}$ and the $\mathcal J$-part is equal to zero.
	\item Degree two: $\alpha(a.b)=[a,b]=\underbrace{s(a.b)}_{\in\mathcal J}-\underbrace{t(a.b)}_{\in \mathcal V}$. The torsion $t=\rho\circ\alpha$ belongs to $T_2^1(\mathcal V)$, the curvature $R_\alpha=R$ belongs to $T_3^1(\mathcal V)$.
\item Degree three:
$$
	\alpha(a.b.c)=\big[[a,b],c\big]=\underbrace{\llbracket s(a.b),
	c\rrbracket +s\big((c\rhd a).b\big)+s\big(a.(c\rhd b)\big) - s\big(t(a.b).c\big)}_{\in\mathcal J}
$$
\begin{equation}
\label{vj-degthree}
	\underbrace{-s(a.b)\rhd c+(c\rhd t)(a.b)+t\big(t(a.b).c\big)}_{\in \mathcal V}.
\end{equation}
\end{itemize}

\subsection{Special polynomials}
\label{ssec:specialpoly}

We use the language of operads here. Let us recall that a $\mathbb S$-module $\mathcal P$ is a collection $(\mathcal P_n)_{n\ge 0}$ of modules over some base commutative unital ring, together with a right action of the symmetric groupoid $\mathbb S=\bigsqcup_{n\ge 0} S_n$, i.e., a right action of the symmetric group $S_n$ on $\mathcal P_n$ for each $n\ge 0$. An operad is a $\mathbb S$-module $\mathcal P$ together with global compositions
$$	
	\gamma:\mathcal P_n\otimes\mathcal P_{k_1}\otimes\cdots\otimes \mathcal P_{k_n}
	\to {\mathcal P}_{k_1 + \cdots + k_n}
$$
functorial with respect to symmetric group actions, and subject to associativity and unitality axioms \cite{LV, Me}. We denote by $\mathcal P(\mathcal V)$ the operad of $\mathbb R$-multilinear maps\footnote{The standard notation in the literature is $\mop{End}\mathcal V$ or $\mop{Endop}\mathcal V$.} on $\mathcal V$, namely $\mathcal P_n(\mathcal V)=\mop{Hom}_{\mathbb R}({\mathcal V}^{\otimes n},\mathcal V)$. The symmetric groups act on the right by permuting the variables, and the compositions $\gamma$ are obviously defined. The unit for the composition is $\mop{Id}_{\mathcal V}\in\mathcal P_1(\mathcal V)$.\\

A monomial of degree (also named arity) $n\ge 0$ is a nonzero element of $\mathcal P_n(\mathcal V)$. A polynomial is a finite sum of monomials, possibly of different arities.

\begin{definition}$\strut$
\label{def-pol}
\begin{itemize}
	\item A geometrically special polynomial \cite{Gavrilov2008b} is a polynomial $\omega$ for which there exists a $\mathcal R$-linear map $\widetilde\omega:\mathcal A\to \mathcal V$ such that $\omega=\widetilde\omega\circ\pi$, where $\pi$ is the natural projection from $T_{\mathbb R}(\mathcal V)$ onto $\mathcal A$. 
	\item A special polynomial \cite{Gavrilov2006} is a polynomial made, by means of iterated compositions, of derivatives of torsion and curvature, possibly permuted.
	\item A polynomial of Lie type is a polynomial made, by means of iterated compositions, of derivatives of $t_\alpha$'s and $R_\alpha$'s (where $\alpha$ is a Lie monomial), possibly permuted.
\end{itemize}
\end{definition}

The corresponding sets are respectively denoted by $\mathcal P_{\mathcal R}(\mathcal V)$, $\mathcal S(\mathcal V)$ and $\mathcal P_{\smop{Lie}}(\mathcal V)$. It is obvious from Definition \ref{def-pol} above that those are three suboperads of $\mathcal P(\mathcal V)$. More precisely, $\mathcal S(\mathcal V)$ is the suboperad generated by $\{\nabla^nt,\nabla^nR,\,n\ge 0\}$, and $\mathcal P_{\smop{Lie}}(\mathcal V)$ is the suboperad generated by $\{\nabla^nt_\alpha,\nabla^nR_\alpha,\,n\ge 0\, , \alpha\hbox{ Lie monomial}\}$. For later use, we define an extended version:

\begin{definition}\label{def-pol-ext}
An extended geometrically special polynomial \cite{Gavrilov2008b} is a finite sum of monomials $\omega:\mathcal V^{\otimes n}\to\mathcal{DM}$ for which there exists a $\mathcal R$-linear map $\widetilde\omega:\mathcal V^{\otimes n}_{\mathcal R}\to \mathcal {DM}$ such that $\omega=\widetilde\omega\circ\pi$.
\end{definition}
The higher-order covariant derivative $\rho:\mathcal V^{\otimes n}\to\mathcal{DM}$ is an example of extended geometrically special monomial.\\

The inclusion $\mathcal S(\mathcal V)\subset \mathcal P_{\mathcal R}(\mathcal V)$ holds, i.e., any special polynomial is geometrically special, and the reciprocal is conjectured \cite[Section 7]{Gavrilov2006}. The inclusion $\mathcal P_{\smop{Lie}}(\mathcal V)\subset \mathcal P_{\mathcal R}(\mathcal V)$ is obvious from Proposition \ref{lem:LieMon}. We partly answer to Gavrilov's conjecture as follows:

\begin{theorem}
\label{thm:special}
$\mathcal P_{\smop{Lie}}(\mathcal V)= \mathcal S(\mathcal V)$.
\end{theorem}


\begin{proof}
The inclusion $\mathcal S(\mathcal V)\subset {\mathcal P}_{\smop{Lie}}(\mathcal V)$ is obvious. In order to show the reverse inclusion, it suffices to prove that $t_\alpha$ and $R_\alpha$ are special polynomials for any Lie monomial $\alpha$. We proceed by induction on the degree: the claim is obvious in degrees one and two, and for $\alpha(a.b.c)=\big[[a,b],c\big]$ we have from \eqref{vj-degthree}:
$$
	t_\alpha(a.b.c)=R(a.b.c)-(c\rhd t)(a.b)-t\big(t(a.b).c)\big)
$$
and
$$
	R_\alpha(a.b.c.d)=-(c\rhd R)(a.b.d)-R\big(t(a.b).c.d\big),
$$
which proves the claim. By an iterated use of the Jacobi identity, any Lie monomial of degree $n+1$ can be written as a linear combination of Lie monomials of the form
$$
	\alpha(x_1\cdots x_{n+1})=[x_j,\beta(x_1\cdots\widehat{x_j}\cdots x_{n+1})]
$$
where $\beta$ is a Lie monomial of degree $n$. For example, consider the equality
$$
	\big[[a,b],[c,d]\big]=-\Big[d,\big[[a,b],c\big]\Big]+\Big[c,\big[[a,b],d\big]\Big].
$$
We therefore compute, with $X:=x_1\cdots\widehat{x_j}\cdots x_{n+1}$:
\begin{eqnarray*}
	\lefteqn{\alpha(x_1\cdots x_{n+1})
	=[x_j,\beta(x_1\cdots\widehat{x_j}\cdots x_{n+1})]}\\
	&=&\llbracket x_j,\beta(X) \rrbracket-x_j\rhd \beta(X)+\beta(X)\rhd x_j\\
	&\stackrel{\eqref{Dtensor1}}{=}&\llbracket x_j,\beta(X) \rrbracket-\beta(x_j\rhd X)+\beta(X)\rhd x_j\\
	&=&-\llbracket x_j,t_\beta(X) \rrbracket+t_\beta(x_j\rhd X)-t_\beta(X)\rhd x_j\\
	&  & \quad +\llbracket x_j,s_\beta(X) \rrbracket-s_\beta(x_j\rhd X)+s_\beta(X)\rhd x_j\\
	&=&-\llbracket x_j,t_\beta(X) \rrbracket+\lbc x_j,t_\beta(X)\rbc+t_\beta(x_j\rhd X)-t_\beta(X)\rhd x_j
-\lbc x_j,t_\beta(X)\rbc\\
	& & \quad +\llbracket x_j,s_\beta(X) \rrbracket-s_\beta(x_j\rhd X)+s_\beta(X)\rhd x_j\\
	&=&-s\big(x_j.t_\beta(X)\big)-t\big(x_j.t_\beta(X)\big)-(x_j\rhd t_\beta)(X)\\
	& & \quad +\llbracket x_j,s_\beta(X) \rrbracket-s_\beta(x_j\rhd X)+s_\beta(X)\rhd x_j
\end{eqnarray*}
from which we get
\begin{equation*}
	t_\alpha(x_1\cdots x_{n+1})=t\big(x_j.t_\beta(X)\big)+(x_j\rhd t_\beta)(X)-R_\beta(X.x_j)
\end{equation*}
and
\begin{eqnarray*}
	R_\alpha(x_1\cdots x_{n+2})
	&=&\Big(\llbracket x_j,s_\beta(X) \rrbracket-s\big(x_j.t_\beta(X)\big)-s_\beta(x_j\rhd X)\Big)\rhd x_{n+2}\\
	&=&x_j\rhd\big(s_\beta(X)\rhd x_{n+2}\big)-s_\beta(X)\rhd(x_j\rhd x_{n+2})\\
	&&-R(x_j.t_\beta(X).x_{n+2})-R_\beta\big((x_j\rhd X.x_{n+2})\big)\\
	&=&x_j\rhd R_\beta(X.x_{n+2})-R_\beta\big(X.(x_j\rhd x_{n+2})\big)\\
	&&-R(x_j.t_\beta(X).x_{n+2})-R_\beta\big((x_j\rhd X).x_{n+2}\big)\\
	&=&(x_j\rhd R_\beta)(X.x_{n+2})-R_\beta\big((x_j\rhd X).x_{n+2}\big),
\end{eqnarray*}
which ends up the induction step and therefore proves the result.
\end{proof}

\section{Gavrilov's double exponential}
\label{sec:double-exp}

Gavrilov's double exponential \cite{Gavrilov2006} is a formal series in two indeterminates, $t$ and $s$, without constant term, which can be explicitly written as follows:
\begin{equation}
\label{eq:double-exp}
	q_*(tv, sw)
	= \beta^{-1} \Big(\text{BCH}\Big\{\beta(tv),\, \beta\big(s\lambda(tv,w)\big)\Big\}\Big).
\end{equation}
Here $v$ and $w$ are two vector fields on a smooth manifold $\mathcal M$ endowed with an affine connection $\nabla$. The notation $\mathrm{BCH}$ refers to the usual Baker--Campbell--Hausdorff series in the completed Lie algebra 
$$
	\overline {\mathcal V}=(t\mathcal{XM}[[s,t]]+s\mathcal{XM}[[s,t]],\lbc.\,,.\rbc),
$$
and $\beta$ stands for Gavrilov's $\beta$-map described earlier in Section \ref{sect:framed-pl}. The map $\lambda$ was introduced in Paragraph \ref{par:Z}. The double exponential \eqref{eq:double-exp} can be informally described in geometrical terms as follows: starting from a point $x \in \mathcal M$ in the direction given by the vector field $v$ at $m$ and following the geodesic $\exp_x^\nabla t'v(x)$ up to time $t'=t$, one reaches the point $y=\exp_x^\nabla t v(x) \in \mathcal M$. Let $W(y)$ be the vector field $w$ at $x$ parallel-transported to the point $y$. Following the geodesic $\exp_y^\nabla t'W(y)$ up to time $t'=s$, one reaches a third point $z=\exp_y^\nabla sW(y)$ on $\mathcal M$. Gavrilov's double exponential permits to express this point following a geodesic starting from $x \in \mathcal  M$
$$
	z=\exp_x^\nabla q_*(tv, sw)(x) \in \mathcal M.
$$
\vskip 3mm
\centerline{\includegraphics[scale=0.5]{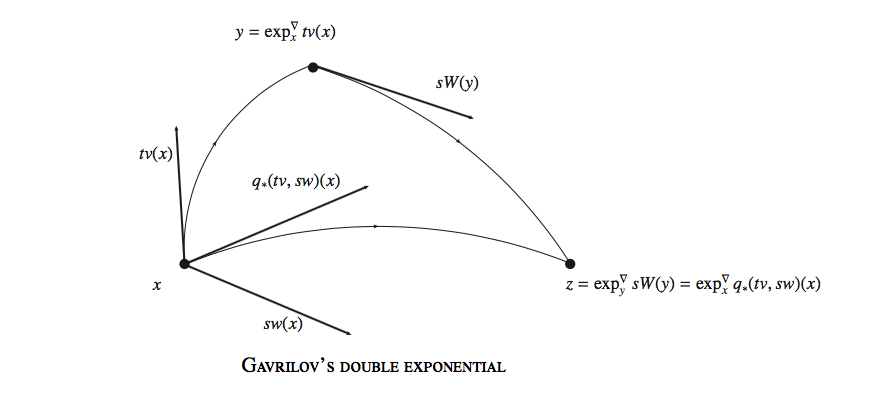}}

\subsection{Heuristic approach}

Let us briefly outline how Formula \eqref{eq:double-exp} can be heuristically obtained from this geometric description. Any tangent vector $u$ at any point $x \in \mathcal M$ gives rise to a vector field on $\mathcal M$ (at least on a sufficiently small neighborhood of $x$) by parallel-transporting $u$ at any point $x'$ along the unique geodesic joining $x$ to $x'$. We denote somewhat abusively by $\beta(u)$ this vector field, and we denote its flow by $\exp t\beta(u)$ or $\exp\beta(tu)$. Let $\delta$ be the unique tangent vector at $x$ such that $x'=\exp^\nabla_x(\delta)$. We denote by $\lambda(\delta,u)$ the parallel transport of $u$ at $x'$ along the geodesic $t\mapsto \exp^\nabla_m(t\delta)$. The picture above can be read as the composition of two flows:
$$
	\exp\beta(tv)\exp\beta(sW)=\exp\beta\big(q_*(tv,sw)\big),
$$
hence
\begin{equation}\label{qbeta}
	\exp\beta(tv)\exp\beta\big(\lambda(tv,sw) \big)=\exp\beta\big(q_*(tv,sw)\big),
\end{equation}
which gives \eqref{eq:double-exp}. The value of the series $q_*(tv,sw)$ at point $x$ indeed depends only on the values at $x$ of the two vector fields $v$ and $w$ \cite[Proposition 4]{Gavrilov2006}. We'll give in Subsection \ref{ssect:double-exp-bis} our own proof of this crucial fact (Remark \ref{rmk:final}).

\subsection{Another expression of the double exponential}\label{ssect:double-exp-bis}

As before, let $\mathcal R=C^\infty(\mathcal M)$, let $\mathcal V=(\mathcal{XM},\rhd,\lbc.\,,.\rbc)$ be the framed Lie algebra of vector fields, and let $(\mathfrak g,\rhd,[.\,,.])=\mop{Lie}_{\mathcal R}(\mathcal V)$ be the post-Lie algebra defined in Paragraph \ref{tc-postLie}. Let $\widehat{\mathcal U(\mathfrak g)}$ be the completion of the enveloping algebra of $\mathfrak g$, endowed with both associative products $\cdot$ and the GL-product, $*$.

\begin{proposition}\label{prop:hans}
Let $v,w\in \mathfrak g$, and let $\widetilde w\in\widehat{\mathfrak g}$ such that $(\exp^{\textstyle\cdot} v)\rhd \widetilde w= w$. Then the following holds in $\widehat{\mathcal U(\mathfrak g)}$:
\begin{equation}
\label{postLieBCHrel}
	\exp^{\textstyle\cdot} v*\exp^{\textstyle\cdot} \widetilde w=\exp^{\textstyle\cdot} z^{\textstyle\cdot}(v,w)
\end{equation}
with
$$
	z^{\textstyle\cdot}(v,w)=\mop{BCH}^{\textstyle\cdot}(v,w)=v+w+\frac 12[v,w]+\frac 1{12}\big[[v,w],w-v\big]+\cdots
$$
\end{proposition}

\begin{proof}
Using that $\exp^{\textstyle\cdot} v$ is grouplike for the coproduct $\Delta_\shuffle$ and that $L_{\exp^{\textstyle\cdot} v}^\rhd=(\exp^{\textstyle\cdot} v)\rhd -$ is an automorphism for the product $\cdot$, we get
\begin{eqnarray*}
	\exp^{\textstyle\cdot} v*\exp^{\textstyle\cdot} \widetilde w
	&=&\exp^{\textstyle\cdot} v\cdot\big((\exp^{\textstyle\cdot} v)\rhd \exp^{\textstyle\cdot} \widetilde w\big)\\
	&=&\exp^{\textstyle\cdot} v\cdot \Big(\exp^{\textstyle\cdot}\big((\exp^{\textstyle\cdot} v)\rhd \widetilde w\big)\Big)\\
	&=&\exp^{\textstyle\cdot} v\cdot\exp^{\textstyle\cdot} w\\
	&=&\exp^{\textstyle\cdot} z^{\textstyle\cdot}(v,w).
\end{eqnarray*}
\end{proof}

\noindent From \eqref{betachirho}, $\beta=\rho\circ\chi\restr{\overline{\mathcal V}}$, and \eqref{qbeta} we get
\begin{equation}
	\rho\Big(\exp^*\chi(tv)*\exp^*\chi\big(\lambda(tv,sw)\big)\Big)=\rho\Big(\exp^*\chi\big(q_*(tv,sw)\big)\Big),
\end{equation}
hence
\begin{equation}
	\rho\Big(\exp^{\textstyle\cdot}(tv)*\exp^{\textstyle\cdot}\big(\lambda(tv,sw)\big)\Big)
	=\rho\Big(\exp^{\textstyle\cdot}\big(q_*(tv,sw)\big)\Big).
\end{equation}
\noindent From Proposition \ref{prop:hans} together with \eqref{lambda}, saying that $\lambda(tv,sw)=\exp^{\ast} (-\chi(tv))\rhd sw$, we therefore get
\begin{equation}
\label{local}
	\rho\Big(\exp^{\textstyle\cdot}\big(z^{\textstyle\cdot}(tv,sw)\big)\Big)
	=\rho\Big(\exp^{\textstyle\cdot}\big(q_*(tv,sw)\big)\Big).
\end{equation}
This in turn yields
\begin{equation}
	\chi\big(z^{\textstyle\cdot}(tv,sw)\big)=\chi\big(q_*(tv,sw)\big) \hbox{ modulo }\mathcal J.
\end{equation}
Bearing in mind that $q_*(tv,sw)$ belongs to $\overline {\mathcal V}$ contrarily to $z^{\textstyle\cdot}(tv,sw)$, and using \eqref{betachirho} again, we finally get

\begin{theorem}
\label{thm:q}
\begin{equation*}
	q_*(tv,sw)=\beta^{-1}\Big(\rho\circ\chi\big(z^{\textstyle\cdot}(tv,sw)\big)\Big).
\end{equation*}
\end{theorem}

\begin{remark}\label{rmk:final}\rm
As already proved by Gavrilov (\cite[main Theorem]{Gavrilov2007} and \cite[Section 4]{Gavrilov2006}), for any $x \in \mathcal M$, the (formal) tangent vector $q_*(tv,sw)(x)$ depends only on the two tangent vectors $v(x)$ and $w(x)$. Our interpretation of this fact in the post-Lie framework is the following: observe that the expression $\rho\big(\exp^{\textstyle\cdot}\big(z^{\textstyle\cdot}(tv,sw)\big)\big)$ is geometrically special in the sense of Definition \ref{def-pol-ext}. In other words, for any function $f \in \mathcal{R}$, the expression $\rho\big(\exp^{\textstyle\cdot}\big(z^{\textstyle\cdot}(tv,sw)\big)\big)f(x)$ depends on the vector fields $v$ and $w$ through $v(x)$ and $w(x)$ alone. Identity \eqref{local} then implies the same for $\rho\big(\exp^{\textstyle\cdot}\big(q_*(tv,sw)\big)\big)f(x)$. As $q_*(tv,sw)$ is a (formal) vector field, we have that
\begin{equation}
\label{a1}
	\rho\big(\exp^{\textstyle\cdot}\big(q_*(tv,sw)\big)\big)f(x) 
	= \rho\big(\exp^{\ast}(Q )\big)f(x)=\mathrm{Exp}(Q)f(x),
\end{equation}
where $Q\in\overline{\mathcal V}$ is the unique geodesic formal vector field such that $Qf(x)=q_*(tv,sw)f(x)$ for any function $f$ in $\mathcal{R}$, and where $\mop{Exp}(Q)\in\mathcal{DM}[[s,t]]$ stands for its formal flow. The geodesic property of $Q$ is expressed as
\begin{equation}
\label{geodesic}
	Q\rhd Q=0,
\end{equation}
from which we get
\begin{equation}\label{geodesic-b}
\exp^{\ast}(Q )=\exp^{\textstyle\cdot}(Q).
\end{equation}
Here we used \eqref{exprelation1} together with the fact that \eqref{geodesic} implies that the inverse post-Lie Magnus expansion reduces to the identity map. From \eqref{local}, \eqref{a1} and \eqref{geodesic-b}, the evaluation $\mathrm{Exp}(Q)f(x)=\rho\big(\exp^{\textstyle\cdot}(Q )\big)f(x)$ of the formal flow at $x$ depends on $v$ and $w$ only through $v(x)$ and $w(x)$. The same is therefore true for
$$
	Qf(x)=\rho\big(\log^{\textstyle\cdot}\exp^{\textstyle\cdot}(Q )\big)f(x).
$$
we finally deduce from $Qf(x)=q_*(tv,sw)f(x)$, that $q_*(tv,sw)f(x)$ depends on $v$ and $w$ through $v(x)$ and $w(x)$ only. 
\end{remark}

\section*{Conclusion}
\label{conclusion}

In this work, we have explored Gavrilov's results in \cite{Gavrilov2006,Gavrilov2008,Gavrilov2008b,Gavrilov2010} from the post-Lie algebra perspective, thus showing the advocating the important role of this notion in differential geometry. This approach should be relevant in even broader contexts, such as Lie algebroids \cite{munthe2020invariant} and Lie--Rinehart algebras \cite{FMKM2021}, their algebraic counterparts.


\appendix

\section{Proofs of Section 2}
\label{app:proofs}

The final computation of the proof of Theorem \ref{thm:Dtensor}:
\allowdisplaybreaks
\begin{align*}
	\lefteqn{\big((XY-YX)U\big)\rhd V=\big(X(YU)\big)\rhd V-\big(Y(XU)\big)\rhd V}\\
	&=X\rhd(YU\rhd V)-(X\rhd YU)\rhd V -(X\leftrightarrow Y)\\
	&=X\rhd\big(Y\rhd(U\rhd V)\big)-X\rhd\big((Y\rhd U)\rhd V\big)\\
	&\hskip 5mm -\big((X\rhd Y)U\big)\rhd V-\big(Y(X\rhd U)\big)\rhd V -(X\leftrightarrow Y)\\
	&=X\rhd\big(Y\rhd(U\rhd V)\big)-X\rhd\big((Y\rhd U)\rhd V\big)\\
	&\hskip 5mm -(X\rhd Y)\rhd(U\rhd V)+\big((X\rhd Y)\rhd U\big)\rhd V\\
	&\hskip 5mm -Y\rhd\big((X\rhd U)\rhd V\big)+\big(Y\rhd(X\rhd U)\big)\rhd V -(X\leftrightarrow Y)\\
	&=XY\rhd(U\rhd V)-YX\rhd(U\rhd V)-\big((XY-YX)\rhd U\big)\rhd V\\
	&=\mathrm{a}_\rhd(XY-YX, U,V),
\end{align*}
which proves Theorem \ref{thm:Dtensor}.

\smallskip

The induction in the proof of Proposition \ref{magma-deshuffle} is given here. The length zero case is trivial and the length one case is the coderivation property mentioned above. Supposing $U=x\cdot U'$ we compute, using the induction hypothesis:
\allowdisplaybreaks
\begin{align*}
	\lefteqn{\Delta_\shuffle(U\rhd V)=\Delta_\shuffle(x U'\rhd V)}\\
	&=\Delta_\shuffle\big(x \rhd (U' \rhd V)-(x \rhd U') \rhd V\big)\\
	&=(x\otimes 1+1\otimes x) \rhd \Delta_\shuffle(U \rhd V) 
		- \Delta_\shuffle(x \rhd U') \rhd \Delta_\shuffle(V)\\
	&=x \rhd (U'_{(1)} \rhd V_{(1)}) \otimes U'_{(2)} \rhd V_{(2)}
		+U'_{(1)} \rhd V_{(1)} \otimes x \rhd (U'_{(2)} \rhd V_{(2)})\\
	&\hskip 8mm -(x\rhd U'_{(1)}) \rhd V_{(1)} \otimes U'_{(2)}\rhd V_{(2)}
		- U'_{(1)}\rhd V_{(1)} \otimes (x \rhd U'_{(2)})\rhd V_{(2)}\\
	&=x U'_{(1)}\rhd V_{(1)}\otimes U'_{(2)}\rhd V_{(2)}
		+U'_{(1)}\rhd V_{(1)}\otimes xU'_{(2)}\rhd V_{(2)}\\
	&=U_{(1)}\rhd V_{(1)}\otimes U_{(2)}\rhd V_{(2)},
\end{align*}
which proves Proposition \ref{magma-deshuffle}.

\smallskip

The induction in the proof of Proposition \ref {prop:pre-glp} is given here. The case $\ell=1$ is just a reformulation of \eqref{Dtensor2}. For $\ell\ge 2$ we can suppose $U=xU'$ where $x\in M$ and where $U'$ is a monomial of length $\ell-1$. We compute
\allowdisplaybreaks
\begin{align*}
\lefteqn{U\rhd(V\rhd W)=xU'\rhd(V\rhd W)}\\
&\stackrel{\eqref{Dtensor2}}{=}x\rhd\big(U'\rhd(V\rhd W)\big) - (x\rhd U')\rhd(V\rhd W)\\
&=x\rhd\Big(\big(U'_{(1)}(U'_{(2)}\rhd V)\big)\rhd W\Big) 
	-(x\rhd U')\rhd(V\rhd W) \hbox{ (from induction)}\\
&\stackrel{\eqref{Dtensor2}}{=}\Big(x U'_{(1)}(U'_{(2)}\rhd V)
	+x\rhd\big(U'_{(1)} (U'_{(2)}\rhd V\big) \\
	&\hskip 8mm  -(x\rhd U')_{(1)}\big((x\rhd U')_{(2)}\rhd V\big)\Big)\rhd W
 \hbox{ (from induction again)}\\
&\stackrel{\eqref{Dtensor1}}{=}\Big(x U'_{(1)} (U'_{(2)}\rhd V)+U'_{(1)} \big(x\rhd(U'_{(2)}\rhd V)\big) - (x\rhd U'_{(2)})\rhd V\Big)\rhd W\\
&\hskip 8mm \hbox{ (from the fact that $L_x$ is a coderivation)}\\
&=\Big(x U'_{(1)} (U'_{(2)}\rhd V)+U'_{(1)} (x U'_{(2)}\rhd V)\Big)\rhd W\\
&=\big(U_{(1)}(U_{(2)}\rhd V)\big)\rhd W,
\end{align*}
which yields \eqref{Dtensor2-iter} and therefore proves Proposition \ref{prop:pre-glp}.

\section{Planar multi-grafting}
\label{ssec:multigrafting}

We use the left grafting representation here (see Remark \ref{lg}). We have $T\big((\mop{Mag}(A),\rhd)\big) = (\mathcal F^{pl}_A,\rhd)$, where $\mathcal F^{pl}_A$ is the linear span of ordered forests of planar rooted trees, and $\rhd$ is extended by means of \eqref{Dtensor1} and \eqref{Dtensor2}. Recall that any planar rooted tree $\tau \in T^{pl}_A$ with the root decorated by $a \in A$ can be written in terms of the so-called $B^a_+$-operator, that is, $\tau= B^a_+[\tau_1 \cdots \tau_n]$, for $\tau_1 \cdots \tau_n \in \mathcal F^{pl}_A$. It adds a root decorated by $a \in A$ and connects it via an edge to every root in the forest $\tau_1 \cdots \tau_n$. For example, denoting the empty tree by $1$, we have
$$
	\Forest{[a]}=B^a_+[1],
	\quad
	\Forest{[a[b]]}=B^a_+[\Forest{[b]}],
	\quad
	\Forest{[a[c[b]]]}=B^a_+[\Forest{[c[b]]}],
	\quad
	\Forest{[a[c][b]]}=B^a_+[\Forest{[c]}\Forest{[b]}].
$$  

We will now consider a multivariate extension of the grafting operation by defining the following brace operations
\begin{equation}
\label{multigraft1}
	\rhd^{\scriptscriptstyle{n+1}}: T_{n+1}\big(\mop{Mag}(A)\big) \times \mop{Mag}(A) \to \mop{Mag}(A).
\end{equation}
Here, $T_{k}\big(\mop{Mag}(A)\big)$ denotes the $k$-th component in the tensor algebra. The multi-grafting in \eqref{multigraft1} is recursively defined for $\tau_1, \tau_2 \in T^{pl}_A$ and a planar forest $\omega$ of length $n$ by 
\begin{equation}
\label{multigraft2}
	(\tau_1 \omega) \rhd^{\scriptscriptstyle{n+1}}\tau_2 
	:= \tau_1 \rhd (\omega \rhd^{\scriptscriptstyle{n}} \tau_2)
		- (\tau_1 \rhd \omega) \rhd^{\scriptscriptstyle{n}} \tau_2.
\end{equation}
\noindent In the case of $\tau_1,\ldots,\tau_n \in T^{pl}_A$ and $\tau_2 =\Forest{[a]} $ this simplifies to
$$
	(\tau_1 \cdots \tau_n) \rhd^{\scriptscriptstyle{n}} \Forest{[a]} 
	= (\tau_1 \cdots \tau_n) \rhd^{\scriptscriptstyle{n}} B_+^a[1]  
	= B_+^a[\tau_1 \cdots \tau_n]. 
$$
Rule \eqref{multigraft2} can be summarised combinatorially as follows: graft the trees $\tau_1, \ldots, \tau_n$ in all possible ways onto the tree $\tau$, (i) excluding the grafting of any $\tau_i$ onto any $\tau_j$ and (ii) when grafting several trees,  $\tau_{j_1}, \ldots, \tau_{j_k}$, $1\le j_1 < \cdots < j_{k} \le n$, onto a vertex $v$ of $\tau$, then they must be grafted to the left of the leftmost edge going out from the vertex $v$ of $\tau$ in such a way, that the order among those trees is preserved. As an example, we consider 
\allowdisplaybreaks
\begin{align*}
	\sigma \tau \rhd^{\scriptscriptstyle{2}} \,\Forest{[c[a][b]]}
	= \ &\Forest{[c[\sigma][\tau][a][b]]}
	+ \Forest{[c[\sigma][a[\tau]][b]]} +\Forest{[c[\sigma][a][b[\tau]]]}
	+ \Forest{[c[\tau][a[\sigma]][b]]} +\Forest{[c[a[\sigma][\tau]][b]]}+ \Forest{[c[a[\sigma]][b[\tau]]]}
	+ \Forest{[c[\tau][a][b[\sigma]]]}+\Forest{[c[a[\tau]][b[\sigma]]]}+\Forest{[c[a][b[\sigma][\tau]]]}.
\end{align*}

\medskip


\end{document}